\documentclass[11pt,reqno]{amsart}
\usepackage[margin=1.2in]{geometry}
\usepackage{mathrsfs}
\usepackage{amsthm}
\usepackage{amsmath}
\usepackage{amssymb}
\usepackage{amsfonts}
\usepackage[dvipdfmx]{graphicx}
\usepackage{latexsym}
\usepackage{comment}
\usepackage{yhmath}
\usepackage{ascmac}
\usepackage{caption}
\usepackage{mathtools}
\usepackage[pagewise]{lineno}
\usepackage{dsfont}
\usepackage{color}
\usepackage{bbm}
\usepackage{bm}

\usepackage[numbers,sort&compress]{natbib}

\def\qed{\hfill $\Box$}

\theoremstyle{plain}

\newtheorem{thm}{Theorem}[section]
\newtheorem{lem}[thm]{Lemma}
\newtheorem{cor}[thm]{Corollary}

\newtheorem*{condi1*}{Condition 1}
\newtheorem*{condi2*}{Condition 2}
\newtheorem*{condi3*}{Condition 3}
\newtheorem*{condi4*}{Condition 4}
\newtheorem{pro}[thm]{Proposition}
\theoremstyle{definition}
\newtheorem{df}[thm]{Definition}

\newtheorem{rem}[thm]{Remark}

\newtheorem*{prf*}{Proof}

\newtheorem*{pf*}{}
\newtheorem*{lem*}{LemmaA}
\newtheorem*{lm*}{LemmaB}
\newtheorem*{stra*}{Strategy for the proof of main result A}

\theoremstyle{plain}
\newtheorem{thmx}{Theorem}

\def\g2{l\ge2}

\def\m2l{\mathcal{L}_2}

\newcommand{\norm}[1]{\ensuremath{ \Vert#1\Vert }} 

\newcommand{\vertiii}[1]{{\left\vert\kern-0.25ex\left\vert\kern-0.25ex\left\vert #1
	\right\vert\kern-0.25ex\right\vert\kern-0.25ex\right\vert}}

\newcommand{\Leb}{\mathrm{Leb}}
\newcommand{\Lip}{{\mathrm{Lip}}}

\newcommand{\Om}{\ensuremath{\Omega}}

\newcommand{\om}{\ensuremath{\omega}}
\newcommand{\lm}{\ensuremath{\lambda}}

\newcommand{\sg}{\ensuremath{\sigma}}
\newcommand{\ep}{\ensuremath{\epsilon}}

\newcommand{\dl}{\ensuremath{\delta}}

\DeclareSymbolFont{bbold}{U}{bbold}{m}{n}
\DeclareSymbolFontAlphabet{\mathbbold}{bbold}

\DeclareMathOperator*{\esssup}{ess\sup}
\DeclareMathOperator*{\essinf}{ess\inf}

\DeclareMathOperator{\diam}{diam}

\newcommand{\cE}{\ensuremath{\mathcal{E}}}
\newcommand{\cF}{\ensuremath{\mathcal{F}}}

\newcommand{\cK}{\ensuremath{\mathcal{K}}}
\newcommand{\cL}{\ensuremath{\mathcal{L}}}

\newcommand{\tcL}{\ensuremath{\tilde{\mathcal{L}}}}

\newcommand{\cQ}{\ensuremath{\mathcal{Q}}}

\newcommand{\cS}{\ensuremath{\mathcal{S}}}

\newcommand{\cW}{\ensuremath{\mathcal{W}}}

\newcommand{\cZ}{\ensuremath{\mathcal{P}}}

\newcommand{\sB}{\ensuremath{\mathscr{B}}}

\newcommand{\sF}{\ensuremath{\mathscr{F}}}

\newcommand{\sL}{\ensuremath{\mathscr{L}}}
\newcommand{\sM}{\ensuremath{\mathscr{M}}}

\newcommand{\CC}{\ensuremath{\mathbb C}} 

\newcommand{\NN}{\ensuremath{\mathbb N}}

\newcommand{\RR}{\ensuremath{\mathbb R}}

\newcommand{\ZZ}{\ensuremath{\mathbb Z}} 

\def\lra{\longrightarrow}
\def\var{\text{{\rm var}}}
\def\Var{\text{{\rm Var}}}
\def\BV{\text{{\rm BV}}}
\def\Leb{\text{{\rm Leb}}}

\title[FCB for Open dynamics]{Functional correlation bound for random Lasota--Yorke maps with holes and its applications to conditional normal approximations}

\author[Juho Lepp\"anen]{Juho Lepp\"anen}
\address{Department of Mathematics, Tokai University, Kanagawa, 259-1292, Japan}
\email{leppanen.juho.heikki.g@tokai.ac.jp}

\author[Yuto Nakajima]{Yuto Nakajima}
\address{Faculty of Science and Engineering, Doshisha University, Kyoto, 610-0394, Japan}
\email{yunakaji@mail.doshisha.ac.jp}

\author[Yushi Nakano]{Yushi Nakano}
\address{Faculty of Science, Hokkaido University, Hokkaido, 060-0810, Japan}
\email{yushi.nakano@math.sci.hokudai.ac.jp}

\keywords{Random Lasota--Yorke maps with holes; Functional correlation bound; Conditional central limit theorem; Stein's method}

\subjclass[2020]{37A50, 37H15, 60F05, 37C30}

\date{}

\begin{document}
\begin{abstract}
This paper investigates the statistical properties of random open dynamical systems generated by families of Lasota--Yorke maps.
Open systems, in which trajectories may escape through `holes', model transient phenomena and present additional difficulties for statistical analysis because the underlying ensemble loses mass over time.
We show that the framework of functional correlation bounds (FCB), originally developed for closed systems, can also be adapted to this random open setting.
The extension requires new ingredients based on Lasota--Yorke type inequalities in order to control the effect of escaping trajectories.
We establish an FCB with exponential decay and combine it with the abstract normal-approximation results of \cite{LNN25,LS20} to obtain a conditional CLT with rates in Wasserstein distance and a conditional functional CLT with a rate in an integral distance over Barbour's class of smooth test functions.
Additionally, we adapt Tikhomirov's method to obtain a bound in Kolmogorov distance for the conditional CLT.
\end{abstract}
\maketitle

\section{Introduction}

Open dynamical systems, where trajectories can escape from the phase space through designated `holes', model transient phenomena in a variety of scientific settings. Their statistical analysis is substantially more delicate than in the closed case because the population of surviving trajectories decreases over time. The aim of this paper is to establish a Functional Correlation Bound (FCB) for a class of random open dynamical systems and to use it to derive quantitative conditional normal approximation results.

The study of statistical limit theorems for dynamical systems, particularly the Central Limit Theorem (CLT), often seeks precise rates of convergence. Stein's method, introduced by Charles Stein in the 1970s, has emerged as a powerful tool for obtaining such rates in diverse settings, including dynamical systems (see e.g.~\cite{HK18, DGS, G, GD, H, HY, He, HeL, HLS20, K, P} and references therein). While significant progress has been made for closed systems, the development of corresponding theories for open systems
requires new ideas to handle the effect of escaping orbits.

Recent advancements in limit theorems for open dynamical systems have been made, notably by the third author and his collaborators \cite{AFGNV}, employing methods based on the Nagaev--Guivarc'h spectral method, itself inspired by earlier work of Collet and Mart\'{\i}nez \cite{ColletMartinez99}. Our approach, in contrast, uses the framework of Functional Correlation Bounds (FCB) for Stein's method in chaotic dynamical systems, developed in a series of works by the first author and his collaborators \cite{L17,HLS20,LS20,LNN25}. 
A key insight of this paper is the discovery that the FCB framework proves remarkably effective even in the setting of open dynamical systems. 
The abstract FCB mechanism remains the same, but the proof of the FCB itself must be adapted to account for mass loss, and our argument relies on Lasota--Yorke type inequality for open transfer operators that is tailored to the FCB setting.

This paper focuses on random open dynamical systems generated by families of Lasota--Yorke maps, a broad class encompassing many chaotic systems, and derives an explicit FCB for these systems. As applications, we establish the following results for random Lasota--Yorke maps by applying abstract theorems for systems satisfying FCB:
\begin{enumerate}
\item a conditional CLT with explicit rates of convergence in both the Wasserstein and Kolmogorov distances; the former follows as an application of the abstract theorem in \cite{L17}, while the latter follows from a new abstract theorem proved in Section~\ref{s:abst} using Tikhomirov's method \cite{T80};
\item a conditional functional CLT with an explicit rate of convergence with respect to an integral distance over Barbour's class of smooth test functions, as an application of the abstract theorem in \cite{LNN25}.
\end{enumerate}
All of these results are new even in the autonomous (deterministic) setting of an open dynamical system generated by iterates of a single Lasota--Yorke map. In contrast, several results on rates of convergence in the quenched CLT and functional CLT for random closed expanding dynamical systems exist in the literature; see, for example, \cite{DH20, LSVW25, H23, DL25, LNN25} and references therein.

The paper is organized as follows. Section~2 introduces the random open Lasota--Yorke setting and the standing assumptions. Section~3 states the main results: Theorem~\ref{FCB} gives the functional correlation bound, and Theorem~\ref{CLT} derives its consequences for conditional normal approximations. Section~4 proves Theorem~\ref{FCB}. Section~5 proves Theorem~\ref{CLT}, apart from the estimate with respect to the Kolmogorov distance, which is reduced to an abstract result and completed in Section~6. Finally, Section~7 verifies the assumptions for the non-deterministic example in Proposition~\ref{ex:1b}.

\section{Setup and standing assumptions}
In this section, we introduce the random open Lasota--Yorke systems considered in this paper and collect the standing assumptions used later.
We first describe the underlying random maps, holes, and refined partitions, which lead to Conditions~1 and~2.
We then pass to the BV framework and the associated closed and open transfer operators, recall the conformal objects, and formulate Conditions~3 and~4.

\subsection{Random open Lasota--Yorke maps}

Let $(\Om,\sF,m)$ be a probability space and let $\sg:\Om\to\Om$ be an invertible, ergodic, and $m$-preserving map. Let $I=[0, 1]$ and consider a family of maps 
\[T_\om:I\to I, \omega\in \Om\]
such that for each $\omega\in \Omega,$ $T_{\omega}: I\rightarrow I$ is surjective.
For each $n\in\mathbb N$ and $\omega\in \Omega$, define $T_\omega^n:I\to I$ by 
\begin{align*}
T_\om^{0}(x)=x,\ T_\om^n(x)=T_{\sg^{n-1}(\om)}\circ\dots\circ T_\om(x)\ \text{for}\ x\in I.
\end{align*}
We assume that for each $\omega\in \Omega,$ there exists a finite partition $\mathcal P_\omega$ of $I$ such that each $P\in \mathcal P_{\omega}$ is an interval, $T_{\omega}(P)$ is an interval, and the restriction $T_{\omega}|_{P}$ is monotonic and continuous.
For $n\in \mathbb N$ let $\mathcal P_{\omega}^{(n)}$ denote the monotonicity partition of $T_\om^n$ on $I$ which is given by 
\begin{align*}
\mathcal P_{\omega}^{(n)}:=\bigvee_{j=0}^{n-1}T_\omega^{-j}\left(\mathcal P_{\sigma^j(\omega)} \right).
\end{align*}
We further assume that for each $\omega\in \Omega,$ the partitions $\mathcal P_{\omega}^{(n)}$ are generating, that is, 
\[\bigvee_{n=1}^{\infty}\mathcal P_{\omega}^{(n)}=\mathscr B_I,\]where $\mathscr B_I$ denotes the Borel $\sigma$-algebra of $I.$
Given $P\in\mathcal P_{\omega}^{(n)}$, we denote by
$$	
T_{\om,P}^{-n}:T_\om^n(P)\lra P
$$ 
the inverse branch of $T_\om^n$ which maps $T_\om^n(x)$ to $x$ for each $x\in P$. 
Below, we always suppose the family satisfies the following: 
\begin{condi1*}$\quad$
\begin{itemize} 
\item {\em Measurability condition}: The skew product map $T: \Omega\times I\rightarrow \Omega\times I: (\omega, x)\mapsto (\sigma (\omega), T_{\omega}(x))$ is measurable with respect to the product $\sigma-$algebra $\mathcal F\otimes \mathscr B_I;$
\item {\em Lipschitz condition}: $\log\#\cZ_\om\in L^1(m)$;
\item {\em Covering condition}: for each $n\in \mathbb N,$ there exists $M(n)\in\NN$ such that for any $\om\in\Om$ and any $P\in\cZ_\om^{(n)}$ we have that $T_\om^{M(n)}(P)=I$;
\item {\em Regularity condition}: there exists $K\geq 1$ such that for each $\om\in\Om$, $P\in\cZ_\om$ we have 
 $T_\om\rvert_P$ is $C^2$ and for any $x\in P,$

\[\frac{|T_\om''(x)|}{|T_\om'(x)|}\leq K;\]

\item {\em Hyperbolicity condition}: there exist $1< \kappa_1\leq \kappa_2<\infty$ and $n_1\in\NN$ such that 
\begin{itemize}
\item $\max_{P\in \mathcal P_{\omega}}\sup_{x\in P}|T_\om'(x)|\leq \kappa_2$ for $m$-a.e. $\om\in\Om$,
\item $\min_{P\in \mathcal P_{\omega}}\inf_{x\in P}|(T_\om^{n_1})'(x)|\geq \kappa_1^{n_1}$ for $m$-a.e. $\om\in\Om$,
\end{itemize}
\item {\em Positive diameter condition}: for each $n\in\NN$ there exists 
$$
\ep_n:=\inf_{\om\in\Om}\min_{P\in\cZ_\om^{(n)}}\diam(P) >0.
$$
\end{itemize}
\end{condi1*}


Let $H \subset \Omega \times I$ be a measurable set with respect to the 
product $\sigma$-algebra $\mathcal{F} \otimes \mathscr B_I$
on $\Omega \times I$ such that
\[
0 < \mathrm{Leb}(H_{\omega}) < 1,
\]
where
 $H_{\omega} \subset I$ is the measurable set in $I$ being
uniquely determined by the condition that
\begin{equation*}
\{\omega\} \times H_{\omega} = H \cap \bigl(\{\omega\} \times I\bigr),
\end{equation*}
and $\mathrm{Leb}$ is the Lebesgue measure on $I$.

For each $\omega \in \Omega$, $n \geq 0$ we define
\begin{equation*}
X_{\omega,n} := \left\{ x \in I : T_\omega^j(x) \notin H_{\sigma^j \omega} \text{ for all } 0 \leq j \leq n \right\} = \bigcap_{j=0}^n T_\omega^{-j} (X_{\sigma^j \omega, 0}).
\end{equation*} We also define
\begin{equation*}
X_{\omega, \infty} := \bigcap_{n=0}^\infty X_{\omega,n} = \bigcap_{n=0}^\infty T_\omega^{-n}(X_{\sigma^n \omega, 0})
\end{equation*}
as the set of points which will never land in a hole under iteration of the maps $T^n_\omega$ for any $n \geq 0$. We call $X_{\omega, \infty}$ the \emph{$\omega$-surviving set}. Note that the sets $X_{\omega,n}$ and $X_{\omega, \infty}$ are measurable.

To formulate the remaining geometric assumptions on the open system, we refine the monotonicity partition according to the surviving set.
Recall that $\mathcal{P}_\omega^{(n)}$ denotes the partition of monotonicity of $T_\omega^n$. 
Let
$
 \widehat{\mathcal{P}}_\omega^{(n)}
$
denote the coarsest partition that is finer than $\mathcal{P}_\omega^{(n)},$ 
and in which each element of $\widehat{\mathcal{P}}_\omega^{(n)}$ is either disjoint from $X_{\omega,n-1}$ 
or fully contained in $X_{\omega,n-1}$. We then introduce the following subcollection:
\begin{equation}\label{eq:0315}
  \mathcal{P}_{\omega,*}^{(n)} := 
    \Bigl\{\, P \in \widehat{\mathcal{P}}_\omega^{(n)} : P \subseteq X_{\omega,n-1} \,\Bigr\}. 
\end{equation}
For each $n\in\NN$, $\omega\in \Omega,$ let $\xi_\om^{(n)}\geq 0$ denote the maximum number of contiguous non-full intervals for $T_\om^n$ in $\mathcal{P}_{\omega,*}^{(n)} ,$
 and define 
\begin{align*}\label{eq: def of F_om^n}
F_\om^{(n)}:=\min_{y\in I}\#\{T^{-n}_\omega (y)\}. 
\end{align*}
Let $h_{\omega}$ denote the number of connected component of $H_{\omega}.$
We further impose the following conditions on our random dynamical systems.
\begin{condi2*}$\quad$
\begin{itemize}
\item {\em Condition for connected components with holes}: $\log h_{\omega}\in L^1(m);$
\item {\em Hole condition}: for $m$-a.e.\ $\om\in\Om$ there exists $P\in\cZ_\om$ with $P\cap H_{\omega}=\emptyset$ such that $T_\om(P)=I$;	
\item {\em Growth condition for inverse points}: 
\[
\displaystyle{\frac{1}{n_1}\int_\Om \log F_\om^{(n_1)} \,dm(\om)>\log\frac{\kappa_2}{\kappa_1}+\int_\Om \log(\xi_\om^{(1)}+2)\, dm(\om)},
\]
where $n_1, \kappa_1,$ and $\kappa_2$ are determined by Hyperbolicity condition in Condition 1.
\end{itemize}
\end{condi2*}

The standing assumptions in this subsection are based on the random Lasota--Yorke framework in \cite[Chapter~1]{AFGV}.
This framework should be viewed as the random counterpart of the deterministic Lasota--Yorke maps with holes investigated by Liverani and Maume-Deschamps \cite{LM}.
Furthermore, in the perturbative small-hole regime, \cite[Theorem~J]{AFGV} extends \cite[Theorem~C]{LM} from single deterministic maps to random maps.

\subsection{Random open transfer operators}
We now turn to the operator-theoretic side of the setup.
We first introduce the relevant BV spaces and the associated closed and open transfer operators, and then state the consequences of \cite{AFGV} used to state Conditions~3 and~4.

Denote by $\mathrm{B}(I)$ the set of all bounded real-valued measurable functions on $I$ and for each $u \in \mathrm{B}(I)$ and each $A \subset I$ set
\[
\mathrm{var}_A(u) := \sup \left\{ \sum_{j=0}^{k-1} |u(x_{j+1}) - u(x_j)| : x_0 < x_1 < \ldots < x_k,\ x_j \in A \text{ for all } k \in \mathbb{N} \right\},
\]
which is the variation of $u$ over $A$. If $A = I$ we denote $\mathrm{var}(u) := \mathrm{var}_I(u)$. We set
\[
\mathrm{BV}(I) := \{ u \in \mathrm{B}(I) : \mathrm{var}(u) < \infty \}
\]
as the set of functions of bounded variation on $I$. Let
\[
\|u\|_\infty := \sup_{x\in I}|u(x)| \qquad \text{and} \qquad \|u\|_{\mathrm{BV}} := \mathrm{var}(u) + \|u\|_\infty.
\]
Then $(\mathrm{B}(I), \|\cdot\|_\infty)$ and $(\mathrm{BV}(I), \|\cdot\|_{\mathrm{BV}})$ are Banach spaces.
\begin{rem}\label{rem:BV-convention}
Throughout the paper we use the norm
$
\|u\|_{\mathrm{BV}}=\var(u)+\|u\|_\infty.
$
This is also the convention used in the Chapter~1 results of \cite{AFGV} that we invoke below.
In particular, we do not switch to an $L^1$-based BV norm in the sequel.
Whenever an $L^1$ estimate is needed, we use the elementary inequality
\begin{equation}\label{eq:BV-L1}
\|u\|_{\mathrm{BV}}
=
\var(u)+\|u\|_\infty
\le 2\var(u)+\|u\|_1,
\qquad u\in \mathrm{BV}(I),
\end{equation}
which follows from $\|u\|_\infty\le \var(u)+\|u\|_1$.
\end{rem}
In the following, given a function $u : \Omega \times I \to \mathbb{R}$, 
by $u_\omega : I \to \mathbb R$ we mean
\[
u_\omega(\cdot) := u(\omega, \cdot).
\]

Let $\mathrm{BV}_\Omega(I)$ denote the collection of all functions $u : \Omega \times I \to \mathbb{R}$ satisfying
\begin{itemize}
    \item[(i)] $u_\omega \in \mathrm{BV}(I)$ for each $\omega \in \Omega$;
    \item[(ii)] for each $x \in I$ the function $\Omega \ni \omega \mapsto u_\omega(x)$ is measurable;
    \item[(iii)] the function $\Omega \ni \omega \mapsto \|u_\omega\|_{\mathrm{BV}}$ is measurable.
\end{itemize}

\medskip


For $\omega\in \Om,$ define the closed transfer operator $\mathcal L_{\omega, 0}$ on ${\rm BV}(I)$ by
\[\mathcal L_{\omega, 0}(u)(x)=\sum_{T_{\omega}(y)=x}\frac{u(y)}{|T_{\omega}^{\prime}(y)|},
\]
where $|T_{\omega}^{\prime}(y)|$ is interpreted as the left derivative at a point for which the derivative is not well-defined.
Inductively, for any $n\in \mathbb N,$ we have 
\[
\mathcal L_{\omega, 0}^n(u)(x)=\sum_{T_{\omega}^n(y)=x}\frac{u(y)}{|(T_{\omega}^n)^{\prime}(y)|}.
\]
By a change of variables, one can easily verify that 
\begin{equation}\label{eq:0115a}
\mathrm{Leb}(\mathcal L_{\omega,0}^n(u)\cdot v)=\mathrm{Leb}(u\cdot v\circ T^n_\omega)
\end{equation}
for any $u,v \in \mathrm{BV}(I)$, $n\ge 1$ and $\omega \in \Omega$,
where $\nu (g)$ denotes the integration of a measurable function $g$ on a probability space with a measure $\nu$.

\medskip

We next introduce the random open transfer operator associated with the hole.
Throughout this paper, we denote by $\widehat{J}$ the indicator function on $J.$
For any $\omega\in \Om,$ define $\mathcal L_{\omega}$ on $\mathrm{BV}(I)$ by \[\mathcal L_{\omega} u=\mathcal L_{\omega, 0}(u\widehat X_{\omega, 0}).\]
Inductively, for any $n\in \mathbb N,$ we have 
$\mathcal L_{\omega}^n u=\mathcal L_{\omega, 0}^n (u \widehat X_{\omega, n-1})$, and hence, it follows from \eqref{eq:0115a} that
\begin{equation}\label{eq:0115b}
\Leb(\hat X_{\sigma^n\omega,0}\mathcal L_{\omega}^n u\cdot v)=\Leb(\mathcal L_{\omega, 0}^n (u \widehat X_{\omega, n})\cdot v)
=\Leb(u \widehat X_{\omega, n} \cdot v\circ T_\omega^n)
\end{equation}
for each $u,v \in \mathrm{BV}(I)$
since $\mathcal L_{\omega,0}^n(u_1\cdot u_2\circ T_\omega^n)=u_2\mathcal L_{\omega,0}^n(u_1)$.
Then Conditions 1 and 2 imply the following by \cite{AFGV}:
\begin{thm}
\label{lem0916}$\empty$
\begin{enumerate}
\item \cite[Theorem A(1)]{AFGV}: There exists a unique random probability measure $\nu_\infty$ supported in $\bigcup_{\omega\in \Omega}\{\omega\}\times X_{\omega, \infty}$ such that 
\begin{align*}
\nu_{\sg\om,\infty}(\cL_\om u)=\lm_\om\nu_{\om,\infty}(u),
\end{align*}
for each $u\in \mathrm{BV}(I)$, where 
\begin{align*}
\lm_\om:=\nu_{\sg\om,\infty}(\cL_\om\widehat I).
\end{align*} 	
Furthermore, we have that 
$\log\lm_\om\in L^1(m);$

\item \cite[Theorem A(2), Proposition 1.10.8 and Lemma 1.11.1]{AFGV}: There exists a function $\phi\in \mathrm{BV}_\Omega(I)$ such that $\nu_{\omega,\infty}(\phi_\omega)=1$ and $\Leb (\phi_\omega )>0$ for $m$-a.e. $\om\in\Om$ we have 
\begin{align*}
\cL_\om \phi_\om=\lm_\om \phi_{\sg(\om)}.
\end{align*}
Furthermore, $\omega\mapsto ||\phi_\omega||_{\mathrm{BV}}$ and $\omega\mapsto\inf \phi_\omega$ are tempered. 
\item  \cite[Theorem A(4) and Remark 1.11.8]{AFGV}: There exists a random conditionally invariant probability measure $\eta$ absolutely continuous with respect to $\mathrm{Leb}$, which is supported on 
\(\bigcup_{\omega \in \Omega} \{\omega\} \times X_{\omega,0}\), and whose disintegrations are given by
\[
    \eta_\omega(u) := 
    \frac{\Leb\left( \widehat X_{\omega, 0} \phi_\omega u \right)}
         {\Leb\left( \widehat X_{\omega, 0} \phi_\omega \right)}
\]
for all $u \in \mathrm{BV}(I)$.
Furthermore, for any $n\in \mathbb N$ and for $m$-almost every $\omega\in \Omega,$  we have $\eta_{\omega}(X_{\omega, n})=\lambda_{\omega}^n
:=\prod_{i=0}^{n-1}\lambda_{\sigma^{i}\omega}.$  
\item \cite[Theorem B]{AFGV}:
For any $n\in \mathbb N$ and $\omega\in \Omega,$ let 
\begin{equation*}Q_{\omega, n}(u)= (\lambda_\omega^n)^{-1} \mathcal{L}_\omega^n u-\nu_{\omega, \infty}(u)\,\phi_{\sigma^n(\omega)}
\end{equation*}
for $u \in \mathrm{BV}(I)$. Then there exists $\kappa \in (0,1)$
such that 
for each $u \in \mathrm{BV}(I)$ for $m$-a.e.\ $\omega \in \Omega$ and all $n \in \mathbb{N}$ we have
\[
    \left\| Q_{\omega, n}(u) \right\|_\infty
    \le D_u(\omega)\|u\|_{BV} \kappa^n
.\]
with a random variable $D_u:\Omega\to(0,\infty)$.
\end{enumerate}
\end{thm}

Using the conformal measure $\nu_\infty$ from Theorem~\ref{lem0916}(1), we now refine the surviving partition elements into good and bad ones and state the additional geometric assumption needed later.
Recall 
  $\mathcal{P}_{\omega,*}^{(n)}$ from \eqref{eq:0315}.
 We then introduce the following subcollection:
\[
  \mathcal{P}_{\omega,b}^{(n)} := 
    \Bigl\{\, P \in \widehat{\mathcal{P}}_\omega^{(n)} : P \subseteq X_{\omega,n-1} 
        \ \text{and}\ \nu_{\omega, \infty}(P)= 0 \,\Bigr\}, 
  \]
\[
  \mathcal{P}_{\omega,g}^{(n)} := 
    \Bigl\{\, P \in \widehat{\mathcal{P}}_\omega^{(n)} : P \subseteq X_{\omega,n-1} 
       \ \text{and}\ \nu_{\omega, \infty}(P) > 0 \,\Bigr\}.
  \]
 Let $\delta_{\omega, n}:=\min_{P\in \mathcal{P}_{\omega,g}^{(n)}}\nu_{\omega, \infty}(P).$ 
We assume the following conditions on our random dynamical systems.
\begin{condi3*}
$\empty$
\begin{itemize}
\item There exists $0<\theta< 1$ and a constant $C_{\theta}>0$ such that for $m$-almost every $\omega\in 
 \Omega,$ for any $n\in \mathbb N,$ 
\begin{align*}(9+16\xi_{\omega}^{(n)})||(T_{\omega}^{n})^{\prime}||_{\infty}^{-1}\le C_{\theta}\theta^n\ \text{and} \ \esssup_{\omega\in \Omega}\frac{\theta}{\lambda_{\omega}} <1 
\end{align*}
where $\lambda_{\omega}$ is a positive number given in Theorem~\ref{lem0916}(1).
\item For any $n\in \mathbb N,$
\begin{align*}
\esssup_{\omega\in \Omega}(2\xi_{\omega}^{(n)}+1)\delta_{\omega, n}^{-1}<\infty.
\end{align*}
\end{itemize}
\end{condi3*}
\begin{rem}
Condition 3 is primarily technical in nature. 
However, we note that in the deterministic case, the second item automatically holds, and the first item holds when the hole size is sufficiently small
 (see \cite[Section 7]{LM}). 
\end{rem}

We now state our final standing assumption, which
strengthens the conclusions of Theorem~\ref{lem0916}.

\begin{condi4*}$\empty$
\begin{itemize}
\item 
There exists a constant $C_\phi \ge 1$ such that for $m$-almost every $\omega\in\Omega$ and every $x\in I$,
\[
C_\phi^{-1}\le \phi_\omega(x)\le C_\phi.
\]
\item There exist constants $D>0$ and $\kappa\in(0,1)$ such that for $m$-almost every $\omega\in\Omega$, every $n\in\mathbb N$ and every $u\in\mathrm{BV}(I)$, 
\[
\|Q_{\omega,n}(u)\|_\infty \le D\|u\|_{\mathrm{BV}}\kappa^n.
\]
\end{itemize}
\end{condi4*}

\begin{rem}
In the deterministic setting, Condition~4 is automatically satisfied by Theorem \ref{lem0916} above, with the exception of the requirement $C_\phi^{-1}\le \phi(x)\le C_\phi$ for every $x\in I$.
The latter is verified by \cite[Theorem~C and Section~7]{LM} for sufficiently small holes.
Consequently, there are abundant deterministic examples satisfying all the conditions of this paper (note that our results are new even in the deterministic case). 
For genuinely non-deterministic perturbative examples satisfying Conditions~1,~2, and~4, we refer the reader to \cite[Section~2.7]{AFGV}.
\end{rem}

\subsection{Non-deterministic example}\label{ss:nondet}
We conclude this section with a simple
non-deterministic example satisfying Conditions~1--4.

Let $\Omega=\{0,1,2,3\}^{\mathbb Z}$ be equipped with the product probability
measure, and let $\sigma:\Omega\to\Omega$ be the left shift. For
$\omega=(\omega_n)_{n\in\mathbb Z}\in\Omega$, define
\[T_\omega(x)=4x \ \text{mod}\ 1, \quad x\in I=[0,1).\]
The map is deterministic, while the randomness is introduced
through the hole:
For $j=0,1,2, 3$, let
\[I_j=[j/4,(j+1)/4).\]
For $\omega=(\omega_n)_{n\in\mathbb Z}\in\Omega$, let
\[\mathcal P_\omega=\{I_0,I_1,I_2,I_3\},\qquad
H_\omega=I_{\omega_0}.
\]
\begin{pro}\label{ex:1b}
The above 
random
open dynamical system 
satisfies Conditions~1--4.
\end{pro}
We postpone the proof of Proposition \ref{ex:1b} to Section \ref{s:ex1b}.

\section{Main results}
Building on the setup of Section 2, where Conditions 1–4 are introduced, we now fix the base family of probability spaces in the random environment and specify the sequence of random variables that form the focus of our main results.
Let $\{\zeta_{\omega}\}_{\omega\in\Omega}$ be a random probability measure, that is, a family of Borel probability measure on $I$ such that
for every $A\in\sB _I$, the map $\Om\ni\om\longmapsto\zeta_\om(A)\in [0,1]$ is measurable.
To state the main results, we restrict the initial random probability measures to those arising from admissible random densities of bounded variation, given by
\begin{multline*}
\mathrm{BV}_{\Omega}(I)_1^+ := \big\{ \psi \in \mathrm{BV}_{\Omega}(I) : 
   \psi_\omega \geq 0, \,
   \nu_{\omega,\infty}(\psi_\omega) = 1, \, 
   \text{$\omega\mapsto ||\psi_\omega||_{\mathrm{BV}}$ is tempered},\\
   \text{and }  \Leb(\psi_\omega) >0 
   \ \text{for } m\text{-a.e. } \omega \in \Omega \big\}.
\end{multline*}
For $\psi\in \mathrm{BV}_{\Omega}(I)_1^+$, we say that
a random probability measure $\zeta$ given by
\[
    \zeta_\omega(u) := 
    \frac{\Leb\left( \widehat X_{\omega, 0} \psi_\omega u \right)}
         {\Leb\left( \widehat X_{\omega, 0} \psi_\omega \right)}
         \qquad (\omega\in\Omega, \; u \in \mathrm{BV}(I))
\]
is a \emph{random probability measure associated with $\psi$.}
An important example is the random conditionally invariant probability measure $\eta$ from Theorem~\ref{lem0916}(3), which is associated with the random density $\phi\in \mathrm{BV}_{\Omega}(I)_1^+$ from Theorem~\ref{lem0916}(2).

Fix a random probability measure $\zeta$ associated with $\psi\in \mathrm{BV}_{\Omega}(I)_1^+$.
We assume that
\[
\zeta_{\omega}(X_{\omega, n})>0\quad \text{for all $n\in \mathbb N$ and $m$-a.e.~$\omega\in\Omega$}.
\]
The conditional probability measure of $\zeta_\omega$ given $X_{\omega, n}$ is then defined by
\begin{align}
\label{measzeta}
\zeta_{\omega, n}(A)=\frac{\zeta_{\omega}(A\cap X_{\omega, n})}{\zeta_{\omega}(X_{\omega, n})}
\end{align}
for each $\omega\in\Omega$, $n\in \mathbb N$ and measurable set $A\subset I$.

Let $f$ be a random observable satisfying that 
\[
L:=2\esssup_{\omega\in\Omega}||f_\omega||_\alpha <\infty
\]
with some $\alpha \in (0,1]$, where $|| h ||_\alpha$ denotes the $\alpha$-H\"older norm of a function $h : I \to \RR$,
$$
\|h\|_\alpha = \|h\|_\infty +
\sup_{x\neq y}\frac{|h(x)-h(y)|}{|x-y|^\alpha}.
$$
Our interest lies in establishing conditional limit theorems for the sequence of random variables $\{f_{\sigma^n\omega} \circ T^n_\omega\}_{0\le n<N}$, conditioned on $X_{\omega,N}$.
Accordingly, the main object of this paper is the sequence of \emph{centered}\footnote{Namely, $\zeta_{\omega,N}(\overline{f_{\omega,N,n}})=0$.} random variables
\[
\overline{f_{\omega,N,n}}(x)
:=f_{\sigma^n\omega} \circ T^n_\omega (x) - \zeta_{\omega,N}( f_{\sigma^n\omega} \circ T^n_\omega ) \quad (0\le n <N),
\]
defined on the probability space $(X_{\omega,N},\mathscr B_I\vert_{X_{\omega,N}}, \zeta_{\omega,N})$, for $\omega\in\Omega$ and $0\le n<N$, where $\mathscr B_I\vert_{X_{\omega,N}}$ is the $\sigma$-field $\mathscr B_I$ restricted to $X_{\omega, N}$.
Notice that $|\overline{f_{\omega,N,n}}(x)|\le L$ for every $N\ge 1$, $0\le n<N$, $x\in I$ and $m$-a.e.~$\omega\in\Omega$.

	Given a function $g : [-L,L]^k \to \RR$, where $k \geq 1$, we define 
\[
\text{Lip}(g) = \max_{1 \leq i \leq k} \sup_{x \in [-L,L]^k} 
\sup_{a \neq a'} 
\frac{|g(x(a/i)) - g(x(a'/i))|}{|a - a'|}.
\]
Here, $x(a/i) \in [-L, L]^k$ denotes the vector obtained from $x \in [-L,L]^k$ by replacing the $i$th component $x_i$ with $a \in [-L, L]$. 
We say that $g$ is separately Lipschitz continuous if $\text{Lip}(g) < \infty$, and we set 
 $\|g\|_{\mathrm{Lip}} = \|g\|_\infty + \text{Lip}(g)$.

The following theorem establishes a functional correlation bound for $\{\overline{f_{\omega,N,n}}\}_{0\le n<N}$ with an exponentially decaying rate function (see Section \ref{s:abst} for its precise definition), \emph{uniformly in $(\omega,N)$}, and constitutes the first main result of this paper.

\begin{thmx}
\label{FCB}
Suppose that Conditions 1-4 hold. Then, there exist $\Omega_1 \subset \Omega$ with 
$m(\Omega_1) = 1$ and $0< r< 1$ such that the following holds for any $\omega \in \Omega_1$.
Let $k\ge2$ and $0\le n_1\le\cdots\le n_k<N$ be integers. Assume that $\{1,\dots,k\}$ is partitioned into consecutive blocks
$$
\{l_{i}+1,\dots,l_{i+1} \}, \quad 0\le i\le p,
$$
where
$$
0=l_0<l_1<\cdots<l_p<l_{p+1}=k \quad \text{and} \quad 
\text{$n_j<n_{j+1}$ if $l_i+1\le j<l_{i+1}$, $0\le i\le p$.}
$$
Then, for any separately Lipschitz continuous function $g:[-L,L]^k\to\mathbb R$, we have
\begin{align*}\label{corre}
&\left|\int G_{\omega ,N}(x,..., x)\ d \zeta_{\omega, N}(x)-\int\cdots \int G_{\omega ,N} (x_1,...,x_{p+1})\ d\zeta_{\omega, N}(x_1)\cdots d\zeta_{\omega, N}(x_{p+1})\right|\\
&\le  \mathbf{C}_\omega L \Vert g \Vert_{ \mathrm{Lip} } \sum_{i=1}^{p}r^{n_{l_i+1}-n_{l_i}},
\end{align*}
for some 
random variable $\mathbf{C}_\omega > 0$ independent of 
$N$, $f$, $L$, 
$n_j$, $l_i$ and $p$,
where
\[
G_{\omega ,N}(x_1,x_2,..., x_{p+1})
:=g\big(
\overline{\bm f_{1,\omega,N}}(x_1),\overline{\bm f_{2,\omega,N}}(x_2),\cdots,\overline{\bm f_{p+1,\omega,N}}(x_{p+1})
\big)
\]
with $\overline{\bm f_{i,\omega,N}}(x):=(\overline{f_{\omega,N,n_{l_{i-1}+1}}}(x),\dots,\overline{f_{\omega,N,n_{l_i}}}(x))$.
\end{thmx}
One can find an explicit form of $\mathbf C_\omega$ in \eqref{eq:0314c1}.

\medskip

As previously mentioned, the uniform functional correlation bound in Theorem \ref{FCB} may lead to various conditional limit theorems for $\{f_{\sigma^n\omega} \circ T^n_\omega\}_{0\le n<N}$ conditioned by $X_{\omega,N}$.
In this paper, we apply the bound to derive several conditional normal approximations.

For Theorem~\ref{CLT} below, we specialize to the case $\zeta_\omega = \eta_\omega$, the random conditionally invariant probability measure from Theorem~\ref{lem0916}(3).
The specialization is also essential for our proof, since the variance growth estimate (Proposition~\ref{prop:0625b}) relies on the conditional invariance property $\eta_{\omega,N}(g\circ T^n_\omega) = \eta_{\sigma^n\omega, N-n}(g)$, which holds for $\eta$ but not for general $\zeta$.

We assume that $f$ belongs to the class $\mathrm{BV}_\Omega(I)$.
We further assume that $f$ is not a coboundary (with respect to $\eta_\infty$), that is,
there does not exist $\tilde f\in L^2(\eta_{\infty})$ 
such that $f_\omega (x)=\tilde f _\omega (x)- \tilde f_{\sigma\omega}\circ T_\omega (x) +\eta_{\omega,\infty}(f_\omega)$ for $m$-a.e.~$\omega\in\Omega$ and $\eta_{\omega,\infty}$-a.e.~$x\in X_{\omega,\infty}$.
Then, by Proposition \ref{prop:0625b}, there exist a constant $C=C_f>0$ and an integer $N_\omega=N_{\omega,f}$ such that
\[
\sigma_{\omega,N} :=\sqrt{\eta_{\omega,N}\left[\left(\sum_{n=0}^{N-1}\overline{f_{\omega,N,n}}\right)^2\right]} 
\]
is larger than $C\sqrt N$ for $m$-a.e.~$\omega\in\Omega$ and all $N\ge N_\omega$.
We set 
\[
W_{\omega,N} :=\sigma_{\omega,N}^{-1}\sum_{n=0}^{N-1}\overline{f_{\omega,N,n}}.
\]
We consider $W_{\omega,N}$ as a random variable on $(X_{\omega,N},\mathscr B_I\vert_{X_{\omega,N}}, \eta_{\omega,N})$.

For a random variable $X$, we write $\mathcal L(X)$ for its law.
The Wasserstein distance $d_{\cW} ( \mu, \nu ) $ and 
the Kolmogorov distance $d_{\cK} ( \mu, \nu ) $ 
between probability
measures $\mu$ and $\nu$ on $\mathbb R$ are defined by
$$
d_\ast(\mu,\nu)
= \sup_{h\in\ast}\left| \int h\, d\mu - \int h\, d\nu \right|,
$$
where
$$
\mathcal W = \{h:\mathbb R\to\mathbb R:\ |h(x)-h(y)|\le |x-y|
\text{ for all } x,y\in\mathbb R\},\quad
\mathcal K = \{\mathbf 1_{(-\infty,x]} : x\in\mathbb R\}.
$$
With a slight abuse of notation, we write
$$
d_\ast(X,Y) = d_\ast\bigl(\mathcal L(X), \mathcal L(Y)\bigr)
$$
for the distance between the laws of two random variables $X$ and $Y$.

Furthermore, we set
\begin{align*}
	\sigma^2_{\omega, N, n} = \sum_{i=0}^{n-1} \sum_{j=0}^{N-1} \eta_{\omega, N}( 
	\overline{f_{\omega,N,i}} \cdot \overline{f_{\omega,N,j}}
	), \quad \mathscr W_{\omega,N}(t) = \sigma_{\omega,N}^{-1} \sum_{n = 0}^{  N - 1 } 
	J_{  \sigma_{\omega, N, n}^2 / \sigma_{\omega,N}^2  }(t)
	\overline{f_{\omega,N,n}}
\end{align*}
for each $t \in [0,1 ]$, where 
\begin{align*}
	J_{ \alpha }(t) = \begin{cases}
		1, & \text{if $t \ge \alpha$}, \\ 
		0, & \text{if $t < \alpha$.}
	\end{cases}
\end{align*}
Denote by $D$ the space of all c\`{a}dl\`{a}g functions $w : [0,1] \to \mathbb{R}$, equipped with the sup norm $\|w\|_\infty = \sup_{t \in [0,1]} |w(t)|$. 
We write $\mathscr M_0$ for a class of functions $h : D \to \mathbb{R}$ introduced in \cite{B90}; its precise definition is given in Section \ref{s:abstr}.  
Roughly speaking, $\mathscr M_0$ consists of twice Fr\'echet differentiable functions whose second derivative satisfies a suitable smoothness condition (see \eqref{eq:smooth}). The space $\mathscr M_0$ is equipped with a norm 
$\|h\|_{\mathscr M}$ whose definition is given in \eqref{eq:norm_m0}.
We remark that $\mathscr M_0$ is sufficiently large in the following sense: a sufficiently fast decay of
\begin{align}\label{eq:smooth_dist}
d_{\cS }(\mathscr W_{\omega,N},\mathscr Z)
:= \sup\Bigl\{\bigl|\textstyle\int h( \mathscr W_{\omega,N} ) \, d\zeta_{\omega, N} - 
E[ h( \mathscr Z ) ]
\bigr| : h\in \mathscr M_0,\ \|h\|_{\mathscr M}\le 1\Bigr\},
\end{align}
such as the one established in Theorem \ref{CLT}, implies that the law of $\mathscr W_{\omega,N}$ converges weakly to the law of $\mathscr Z$,
provided that $N^{-1} \sigma_{\omega, N}^2$ converges to a positive limit; see Theorem \ref{thm:clt_tikhomirov} for details.

We are now in a position to state our second main result, which concerns conditional normal approximations:
\begin{thmx}
\label{CLT}
Suppose that Conditions 1-4 hold and that 
$f \in BV_\Omega (I)$ is not a coboundary.
Then, there exists a positive random variable $\mathcal C$ such that for $m$-a.e.~$\omega\in\Omega$ and every $N\ge N_\omega$, with $N_\omega$ as in Proposition \ref{prop:0625b},
\begin{align*}
&d_{\mathcal W}(W_{\omega,N}, Z)\le \mathcal C_\omega N^{-\frac{1}{2}}, 
\\
&d_{\cK}( W_{\omega,N},Z ) \le \mathcal C_\omega  \log^2(N + 1) N^{-\frac12},
\\
&d_{\cS}(\mathscr W_{\omega,N},\mathscr Z)
\le \mathcal C_\omega N^{-\frac{1}{2}}.
\end{align*}
Here,  $Z$ is a standard normal random variable and $\mathscr Z$ is a standard Brownian motion.
\end{thmx}
We will prove Theorem \ref{CLT} as a consequence of more abstract normal approximation results under a functional correlation bound (Theorem \ref{thm:clt_tikhomirov}), using the functional correlation bound established in Theorem \ref{FCB}.
Section \ref{s:4} is devoted to the proof of Theorem \ref{FCB}, while Section \ref{s:abst} gives the proof of Theorem \ref{CLT}.


\section{Proof of Theorem A}\label{s:4}

By the induction argument on $p$ in \cite[Proposition 2.8]{LNN25}, it suffices to prove Theorem \ref{FCB} for the case $p=1$.
We therefore set $p=1$ and $l_1=l$ in Theorem \ref{FCB}, and establish
\begin{equation}\label{eq:0314c}
\Bigg|
\int G_{\omega, N}(x,x)\,d\zeta_{\omega, N}(x)
\;-\;
\iint G_{\omega, N}(x,y)\,d\zeta_{\omega, N}(x)\,d\zeta_{\omega, N}(y)
\Bigg|
\le  \mathbf{C}_\omega L \Vert g \Vert_{ \mathrm{Lip} } r^{n_{l+1}-n_{l}}
\end{equation}
(recall $L=\esssup_{\omega\in\Omega}||f_\omega||_\alpha$),
where
\[
G_{\omega ,N}(x,y)
=g\big(
\overline{f_{\omega,N,n_{1}}}(x),\cdots,\overline{f_{\omega,N,n_{l}}}(x),\overline{f_{\omega,N,n_{l+1}}}(y),\dots,\overline{f_{\omega,N,n_{k}}}(y)\big).
\]

Let $\psi\in\mathrm{BV}_{\Omega}(I)_1^+$ be the random density with which $\zeta$ is associated.

\subsection{Preliminary lemmas}
Let
\[
n_\ast := n_l + \Big\lceil \frac{n_{l+1}-n_l}{2}\Big\rceil.
\]
For $\theta$ coming from Condition 3 
let $n_0 \in \mathbb{N}$ such that 
\begin{equation}\label{eq:Ctheta}
 C_\theta \theta^{n_0} < \lambda_{\omega}^{n_0}\quad \text{$m$-a.e.~$\omega\in\Omega$},
\end{equation}
and write 
\[
n_\ast = k n_0 + m \quad \text{with $k,m \in \mathbb{N}$, $0 \leq m \leq n_0-1$.}
\]
Define
\[
\mathcal Q_\omega := \left\{ 
\left( \bigcap_{i=1}^k T_\omega^{-(k-i)n_0-m}(J_i) \right) \cap \widetilde{J}
\;\;:\;\; J_i \in \widehat{\mathcal{P}}_{\sigma^{(k-i)n_0 +m}\omega}^{(n_0)} (i=1,..., k), \ 
\widetilde{J} \in \widehat{\mathcal{P}}_\omega^{(m)}
\right\}.
\]

\begin{lem}\label{lem:rand-claim}
There exists
a constant $C_0>0$ depending only on the random system
$T$ 
such that for $m$-a.e.~$\omega\in\Omega$, any $N\in\mathbb N$ and any $J\in\mathcal Q_\omega$,
there exists $c_J\in J$ satisfying
\[
  \big|G_{\omega, N}(x,y)-G_{\omega, N}(c_J,y)\big|
  \le C_0L\|g\|_{\mathrm{Lip}}\rho^{\,n_{l+1}-n_l}
\]
for all $x\in J$ and $y\in I$, where $\rho:=\kappa_1^{-\alpha/2}\in(0,1)$.
\end{lem}

\begin{proof}
Fix $\omega\in\Omega$, $N\in\mathbb N$ and $J\in\mathcal Q_\omega$.
Let $x,c_J\in J$ and $y\in I$.
By the definition of $G_{\omega,N}$ (it depends on $x$ only through the first $l$ orbit coordinates)
and the separate Lipschitz bound of $g$, we have
\[
\big|G_{\omega,N}(x,y)-G_{\omega,N}(c_J,y)\big|
\le
\|g\|_{\mathrm{Lip}}
\sum_{j=1}^{l}
\big| f_{\sigma^{n_j}\omega}(T_\omega^{n_j}(x))-f_{\sigma^{n_j}\omega}(T_\omega^{n_j}(c_J))\big|.
\]
Since $f_{\sigma^{n_j}\omega}$ is $\alpha$--H\"older with $\|f_{\sigma^{n_j}\omega}\|_\alpha\le L$,
\[
\big| f_{\sigma^{n_j}\omega}(T_\omega^{n_j}(x))-f_{\sigma^{n_j}\omega}(T_\omega^{n_j}(c_J))\big|
\le
L\big|T_\omega^{n_j}(x)-T_\omega^{n_j}(c_J)\big|^\alpha
\le
L\mathrm{diam}\big(T_\omega^{n_j}(J)\big)^\alpha .
\]
Hence
\[
\big|G_{\omega,N}(x,y)-G_{\omega,N}(c_J,y)\big|
\le
L\|g\|_{\mathrm{Lip}}
\sum_{j=1}^{l}\mathrm{diam}\big(T_\omega^{n_j}(J)\big)^\alpha .
\]

By construction of $\mathcal Q_\omega$, for each $1\le j\le l$ the map
$T_{\sigma^{n_j}\omega}^{\,n_\ast-n_j}$ maps $T_\omega^{n_j}(J)$ diffeomorphically onto
$T_\omega^{n_\ast}(J)$. 
Using bounded distortion and the uniform expansion constant
$\kappa_1>1$ from Condition~1, we obtain a constant $\widetilde C_0\ge 1$ (independent of $\omega,N,f,J$) such that
\[
\mathrm{diam}\big(T_\omega^{n_j}(J)\big)
\le
\widetilde C_0\kappa_1^{-(n_\ast-n_j)}\,\mathrm{diam}\big(T_\omega^{n_\ast}(J)\big)
\le
\widetilde C_0\kappa_1^{-(n_\ast-n_j)} .
\]
Therefore,
\[
\big|G_{\omega,N}(x,y)-G_{\omega,N}(c_J,y)\big|
\le
L\|g\|_{\mathrm{Lip}}\widetilde C_0^\alpha
\sum_{j=1}^{l}\kappa_1^{-\alpha(n_\ast-n_j)}.
\]
Set $q:=\kappa_1^{-\alpha}\in(0,1)$. Since $n_1<\cdots<n_l$, we have $n_\ast-n_j\ge n_\ast-n_l$ and hence
\[
\sum_{j=1}^{l}\kappa_1^{-\alpha(n_\ast-n_j)}
\le
\sum_{k=n_\ast-n_l}^{\infty} q^{k}
=
\frac{q^{\,n_\ast-n_l}}{1-q}.
\]
Recalling $n_\ast = n_l + \left\lceil\frac{n_{l+1}-n_l}{2}\right\rceil$, we have
$q^{n_\ast-n_l}\le \rho^{n_{l+1}-n_l}$ with $\rho:=q^{1/2}=\kappa_1^{-\alpha/2}\in(0,1)$.
Collecting constants yields
\[
\big|G_{\omega,N}(x,y)-G_{\omega,N}(c_J,y)\big|
\le
C_0L\|g\|_{\mathrm{Lip}}\rho^{\,n_{l+1}-n_l}
\]
with 
\[
C_0=\frac{\widetilde C_0^\alpha}{1-\kappa _1^{-\alpha}}.
\]
This completes the proof.
\end{proof}

By using Lemma~\ref{lem:rand-claim}, we obtain for $m$-a.e.\ $\omega\in\Omega$ and any $N\in\mathbb N$,
\begin{align*}
&\Bigg|
\int G_{\omega, N}(x,x)\,d\zeta_{\omega,N}(x)
-\sum_{J\in\mathcal Q_\omega}\int \hat J(x) G_{\omega, N}(c_J,x)\, d\zeta_{\omega,N}(x)
\Bigg|\\
&
=\left|
  \sum_{J \in \mathcal Q_\omega}
    \int _J \Big(G_{\omega, N}(x,x)-G_{\omega, N}(c_J,x)\Big)\,  d\zeta_{\omega,N}(x)
\right| 
\le C_0L \Vert g \Vert_{ \mathrm{Lip} } \rho^{n_{l+1}-n_l},
\end{align*}
where $c_J\in J$. Moreover,
\begin{align*}
&\Bigg|
\iint G_{\omega, N}(x,y)\,d\zeta_{\omega,N}(x)\,d\zeta_{\omega,N}(y)
-\sum_{J\in\mathcal Q_\omega} \zeta_{\omega,N}(J)\int G_{\omega, N}(c_J,y)\, 
   d\zeta_{\omega,N}(y)
\Bigg|
\\&
=\left|
  \sum_{J \in \mathcal Q_\omega}
    \int \left(\int_J \Big(G_{\omega, N}(x,y)-G_{\omega, N}(c_J,y) \Big) \, d\zeta_{\omega,N}(x)\right)d\zeta_{\omega,N}(y)
\right|
\le C_0L \Vert g \Vert_{ \mathrm{Lip} }\rho^{\,n_{l+1}-n_l}.
\end{align*}

On the other hand, 
since
$  n_{l+1} - n_* = \lfloor \frac{n_{l+1}-n_l}{2} \rfloor\geq \frac{n_{l+1}-n_l-1}{2}$,
setting $C_\kappa = \kappa^{-1/2}$ yields
\[
  \kappa^{n_{l+1} - n_*} \leq C_\kappa (\sqrt{\kappa})^{n_{l+1}-n_{l}}.
\]
Hence, to obtain the desired estimate \eqref{eq:0314c},
it suffices to show that
\begin{equation}\label{eq:*1}
\begin{split}
\Bigg|
\sum_{J\in\mathcal Q_\omega}\int (\hat J(x) - \zeta_{\omega,\infty}(J)) G_{\omega, N}(c_J,x)\,d\zeta_{\omega,N}(x)
\Bigg|
\le  \left(\sum_{j=1}^4C_j(\omega) \right)\Vert g \Vert_{ \mathrm{Lip} }\kappa^{n_{l+1}-n_\ast}
\end{split}\tag{*}
\end{equation}
and
\begin{equation}\label{eq:*2}
\begin{split}
\Bigg|
\sum_{J\in\mathcal Q_\omega}(\zeta_{\omega,N}(J)-\zeta_{\omega,\infty}(J))\int G_{\omega, N}(c_J,y)\,d\zeta_{\omega,N}(y)
\Bigg|
\le
(C_4(\omega)+C_5(\omega))\Vert g \Vert_{ \mathrm{Lip} }\kappa ^{n_{l+1}-n_\ast},
\end{split}\tag{**}
\end{equation}
where $C_j(\omega)$'s are given in \eqref{eq:0313b1}, \eqref{eq:0313b2}, \eqref{eq:0313b3} and \eqref{eq:0314}:
these inequalities yield \eqref{eq:0314c} with 
\begin{equation}\label{eq:0314c1}
r=\max\{\rho,\sqrt{\kappa}\}\qquad \text{and}\qquad \mathbf{C}_\omega =2C_0+2C_\kappa\sum_{j=1}^5C_j(\omega).
\end{equation}
We first give the proof of \eqref{eq:*1} throughout the rest of this section except the last subsection, where we give the proof of \eqref{eq:*2}, because the proof of \eqref{eq:*2} is similar to and easier than that of \eqref{eq:*1}. 

\subsection{Decomposition for \eqref{eq:*1}}
For each $J\in\mathcal Q_\omega$, let
$
\widetilde G_{\omega, N}(x):=\widetilde G_{\omega, N}(c_J,x)$ be the function that satisfies
$\widetilde G_{\omega, N}(c_J,T_\omega^{\,n_{l+1}}(x))=G_{\omega, N}(c_J,x).$
We need the following auxiliary lemma:
\begin{lem}\label{lemratio}
Let $u \in L^1(I)$. Then for each $J \in \mathcal Q_\omega$ and $m$-a.e.\ $\omega\in\Omega$ we have
\[
\int G_{\omega, N}(c_J,y) u(y)\, d\zeta_{\omega,N}(y)
=
\frac{\Leb\left( (\widetilde G_{\omega, N} \widehat{X}_{\sigma^{n_{l+1}}\omega,N-n_{l+1}})\, 
   \mathcal L_\omega^{\,n_{l+1}}(\psi_\omega u) \right)}
{\Leb\left( (\widehat{X}_{\sigma^{n_{l+1}}\omega,N-n_{l+1}})\, \mathcal L_\omega^{\,n_{l+1}}(\psi_\omega) \right)}.
\]
\end{lem}

\begin{proof}
Since $\zeta_{\omega,N}$ satisfies
\[
\zeta_{\omega,N}(A)
=\frac{\zeta_\omega(A\cap X_{\omega,N})}{\zeta_\omega(X_{\omega,N})},
\]
and $\zeta_\omega$ has density $\psi_\omega$ with respect to Leb, we obtain
\begin{align*}
\int G_{\omega, N}(c_J,y)u(y)\, d\zeta_{\omega,N}(y)
&= \frac{\mathrm{Leb}\big( (\widetilde G_{\omega, N} \circ T_\omega^{n_{l+1}}) \psi_\omega u\widehat X_{\omega,N} \big)}
       {\mathrm{Leb}\big( \psi_\omega \widehat X_{\omega,N} \big)}\\
&=\frac{\mathrm{Leb}\big( (\widetilde G_{\omega, N} \circ T_\omega^{n_{l+1}}) \psi_\omega u ( \widehat{X}_{\sigma^{n_{l+1}}\omega,N-n_{l+1}}\circ T^{n_{l+1}})\widehat{X}_{\omega,n_{l+1}}\big)}
       {\mathrm{Leb}\big( \psi_\omega  ( \widehat{X}_{\sigma^{n_{l+1}}\omega,N-n_{l+1}}\circ T^{n_{l+1}}_\omega)\widehat{X}_{\omega,n_{l+1}}\big)}\\
&= \frac{\mathrm{Leb}\Big( (\widetilde G_{\omega, N} \widehat{X}_{\sigma^{n_{l+1}}\omega,N-n_{l+1}}) \circ T_\omega^{n_{l+1}}\psi_\omega u\widehat X_{\omega,n_{l+1}} \Big)}
       {\mathrm{Leb}\Big( (\widehat{X}_{\sigma^{n_{l+1}}\omega,N-n_{l+1}} \circ T_\omega^{n_{l+1}})\psi_\omega\widehat{X}_{\omega,n_{l+1}}\Big)}\\
&= \frac{\mathrm{Leb}\Big( (\widetilde G_{\omega, N} \widehat{X}_{\sigma^{n_{l+1}}\omega,N-n_{l+1}})\,\mathcal L_{\omega,0}^{\,n_{l+1}}(\psi_\omega u\widehat{X}_{\omega,n_{l+1}})\Big)}
       {\mathrm{Leb}\Big( (\widehat{X}_{\sigma^{n_{l+1}}\omega,N-n_{l+1}})\,\mathcal L_{\omega,0}^{\,n_{l+1}}(\psi_\omega\widehat{X}_{\omega,n_{l+1}})\Big)}
\\
&= \frac{\mathrm{Leb}\Big( (\widetilde G_{\omega, N} \widehat{X}_{\sigma^{n_{l+1}}\omega,N-n_{l+1}})\,\mathcal L_\omega^{\,n_{l+1}}(\psi_\omega u)\Big)}
       {\mathrm{Leb}\Big( (\widehat{X}_{\sigma^{n_{l+1}}\omega,N-n_{l+1}})\,\mathcal L_\omega^{\,n_{l+1}}(\psi_\omega)\Big)}.
\end{align*}
Here, we have used the relation 
\(\mathcal L_{\omega, 0}^m(\widehat{X}_{\omega,m}u) = \widehat X_{\sigma^{m}\omega,0} \mathcal L_\omega^mu.\)
\end{proof}

We set 
\[
(I)_{J}:=\int _JG_{\omega, N}(c_J,x) \, d\zeta_{\omega,N}(x),\quad  (II)_{J}:=\zeta_{\omega,\infty}(J)\int G_{\omega, N}(c_J,y)\, d\zeta_{\omega,N}(y).
\]
Note that the left-hand side of \eqref{eq:*1} is $|\sum_{J\in\cQ_\omega}\big( (I)_J-(II)_J\big)|$.
Recall that
\[
\widetilde{\mathcal L}_\omega^{n}=(\lambda^n_{\omega})^{-1}{\mathcal L}_\omega^{n}
\]
for each $\omega\in\Omega$ and $n\in\mathbb N$.
By Lemma~\ref{lemratio},
we obtain
\begin{align*}
(I)_{J}
&= 
   \frac{\mathrm{Leb}\big(\widetilde G_{\omega, N} \widehat X_{\sigma^{n_{l+1}}\omega,N-n_{l+1}} \,
        \mathcal L_\omega^{\,n_{l+1}}(\psi_\omega \widehat J)\big)}
        {\mathrm{Leb}\big(\widehat X_{\sigma^{n_{l+1}}\omega,N-n_{l+1}} \,
        \mathcal L_\omega^{\,n_{l+1}}(\psi_\omega)\big)}\notag\\
&= 
   \frac{
      \mathrm{Leb}\left(\widetilde G_{\omega, N} \widehat X_{\sigma^{n_{l+1}}\omega,N-n_{l+1}} 
       \widetilde{\mathcal L}_{\sigma^{n_{\ast}}\omega}^{\,n_{l+1}-n_*} 
      \circ \mathcal L_\omega^{\,n_*}(\psi_\omega \widehat J)\right)
   }{
      \mathrm{Leb}\left(\widehat X_{\sigma^{n_{l+1}}\omega,N-n_{l+1}} 
      \widetilde{\mathcal L}_{\sigma^{n_{\ast}}\omega}^{\,n_{l+1}-n_*} 
      \circ \mathcal L_\omega^{\,n_*}(\psi_\omega)\right)
   }
\end{align*}
and
\begin{align*}
\begin{split}
(II)_{J}&=
\zeta_{\omega,\infty}(J)\int G_{\omega, N}(c_J,y)\, d\zeta_{\omega,N}(y) \\
&= 
\zeta_{\omega,\infty}(J)\,
   \frac{\mathrm{Leb}\big(\widetilde G_{\omega, N} \widehat X_{\sigma^{n_{l+1}}\omega,N-n_{l+1}} 
        \mathcal L_\omega^{\,n_{l+1}}(\psi_\omega)\big)}
        {\mathrm{Leb}\big(\widehat X_{\sigma^{n_{l+1}}\omega,N-n_{l+1}} 
        \mathcal L_\omega^{\,n_{l+1}}(\psi_\omega)\big)}.
\end{split}
\end{align*}

Recall  
$Q_{\omega, n}u= \widetilde{\mathcal{L}}_\omega^n u-\nu_{\omega, \infty}(u)\,\phi_{\sigma^n(\omega)}
$
for each $\omega\in \Omega$ and $n\in \mathbb N$.
For each $J\in\mathcal Q_\omega$, let
\begin{align*}
&L_J = \Leb\left(\widetilde G_{\omega, N} \widehat X_{\sigma^{n_{l+1}}\omega,N-n_{l+1}}\phi_{\sigma^{n_{l+1}}\omega}\right)
      \nu_{\sigma^{n_{\ast}}\omega, \infty}\left(\mathcal L_\omega^{n_\ast}(\psi_\omega \widehat{J})\right),\\
&R_J = \Leb\left(\widetilde G_{\omega, N}\widehat X_{\sigma^{n_{l+1}}\omega,N-n_{l+1}}\,
      Q_{{\sigma^{n_{\ast}}\omega},n_{l+1}-n_\ast}\mathcal L_\omega^{n_\ast}(\psi_\omega \widehat{J})\right),
\\
&
L_0 = \Leb\left(\widehat X_{\sigma^{n_{l+1}}\omega,N-n_{l+1}}\phi_{\sigma^{n_{l+1}}\omega}\right)\,
      \nu_{\sigma^{n_{\ast}}\omega, \infty}\left(\mathcal L_\omega^{n_\ast}(\psi_\omega)\right),\\
&R_0 = \Leb\left(\widehat X_{\sigma^{n_{l+1}}\omega,N-n_{l+1}}\,
      Q_{{\sigma^{n_{\ast}}\omega},n_{l+1}-n_\ast}\mathcal L_\omega^{n_\ast}(\psi_\omega)\right).
\end{align*}

Then for each $J\in\mathcal Q_\omega$ it is clear that
\begin{align}
\label{ine4}
(I)_J = \frac{L_J+R_J}{L_0+R_0}. 
\end{align}
Furthermore, since $\nu_{\omega, \infty}(\psi_\omega)=1$,
\begin{align*}
\frac{L_J}{L_0}
&= \frac{\mathrm{Leb}\left(\widetilde G_{\omega, N} \widehat X_{\sigma^{n_{l+1}}\omega,N-n_{l+1}}  \phi_{\sigma^{n_{l+1}}\omega}\right)
         \nu_{\sigma^{n_{\ast}}\omega, \infty}\bigl(\mathcal L_\omega^{n_*}(\psi_\omega \widehat J\, )\bigr)}
       {\mathrm{Leb}\left(\widehat X_{\sigma^{n_{l+1}}\omega,N-n_{l+1}} \phi_{\sigma^{n_{l+1}}\omega}\right)
         \nu_{\sigma^{n_{\ast}}\omega, \infty}\bigl(\mathcal L_\omega^{n_*}(\psi_\omega)\bigr)}
\\
&= \frac{\mathrm{Leb}\left(\widehat X_{\sigma^{n_{l+1}}\omega,N-n_{l+1}} \phi_{\sigma^{n_{l+1}}\omega} \widetilde G_{\omega, N}\right)}
        {\mathrm{Leb}\left(\widehat X_{\sigma^{n_{l+1}}\omega,N-n_{l+1}} \phi_{\sigma^{n_{l+1}}\omega}\right)}
   \times  \nu_{\omega, \infty}(\psi_\omega \widehat J)
\\
&= \frac{\mathrm{Leb}\left(\widehat X_{\sigma^{n_{l+1}}\omega,N-n_{l+1}} \phi_{\sigma^{n_{l+1}}\omega} \widetilde G_{\omega, N}\right)}
        {\mathrm{Leb}\left(\widehat X_{\sigma^{n_{l+1}}\omega,N-n_{l+1}} \phi_{\sigma^{n_{l+1}}\omega}\right)}
   \times \zeta_{\omega, \infty}(J).
\end{align*}
On the other hand,
since $\widehat X_{\sigma^{n_{l+1}}\omega,N-n_{l+1}}$ and $\widetilde{\mathcal L}_\omega^{\,n_{l+1}}(\psi_\omega)$ are positive functions, we have
\[
\frac{\mathrm{Leb}\left(\widetilde G_{\omega, N} \widehat X_{\sigma^{n_{l+1}}\omega,N-n_{l+1}}
        \,\widetilde{\mathcal L}_\omega^{\,n_{l+1}}(\psi_\omega)\right)}
       {\mathrm{Leb}\left(\widehat X_{\sigma^{n_{l+1}}\omega,N-n_{l+1}} \,
        \widetilde{\mathcal L}_\omega^{\,n_{l+1}}(\psi_\omega)\right)}\le ||g||_{\mathrm{Lip}}.
\]
Hence since $\nu_{\omega, \infty}(\phi_\omega)=\nu_{\omega, \infty}(\psi_\omega)=1$, we have
\begin{align*}
&\left| \frac{L_J}{L_0}-(II)_{J} \right|\\
&= \zeta_{\omega, \infty}(J)
\left|
\frac{\mathrm{Leb}\left(\widehat X_{\sigma^{n_{l+1}}\omega,N-n_{l+1}} \phi_{\sigma^{n_{l+1}}\omega} \widetilde G_{\omega, N}\right)}
     {\mathrm{Leb}\left(\widehat X_{\sigma^{n_{l+1}}\omega,N-n_{l+1}} \phi_{\sigma^{n_{l+1}}\omega}\right)}
- \frac{\mathrm{Leb}\left(\widetilde G_{\omega, N} \widehat X_{\sigma^{n_{l+1}}\omega,N-n_{l+1}}
        \,\widetilde{\mathcal L}_\omega^{\,n_{l+1}}( \psi_\omega)\right)}
       {\mathrm{Leb}\left(\widehat X_{\sigma^{n_{l+1}}\omega,N-n_{l+1}} \,
        \widetilde{\mathcal L}_\omega^{\,n_{l+1}}(\psi_\omega)\right)}
\right|\\
&
\leq \zeta_{\omega, \infty}(J)\Biggl(
\left|
\frac{\mathrm{Leb}\left(\widehat X_{\sigma^{n_{l+1}}\omega,N-n_{l+1}}\phi_{\sigma^{n_{l+1}}\omega} \widetilde G_{\omega, N}\right)}
     {\mathrm{Leb}\left(\widehat X_{\sigma^{n_{l+1}}\omega,N-n_{l+1}}\phi_{\sigma^{n_{l+1}}\omega}\right)}
- \frac{\mathrm{Leb}\left(\widetilde G_{\omega, N} \widehat X_{\sigma^{n_{l+1}}\omega,N-n_{l+1}}
       \,\widetilde{\mathcal L}_\omega^{\,n_{l+1}}(\psi_\omega)\right)}
       {\mathrm{Leb}\left(\widehat X_{\sigma^{n_{l+1}}\omega,N-n_{l+1}}\phi_{\sigma^{n_{l+1}}\omega}\right)}
\right|
\\
&\qquad +
\left|
\frac{\mathrm{Leb}\left(\widetilde G_{\omega, N} \widehat X_{\sigma^{n_{l+1}}\omega,N-n_{l+1}}
        \,\widetilde{\mathcal L}_\omega^{\,n_{l+1}}(\psi_\omega)\right)}
       {\mathrm{Leb}\left(\widehat X_{\sigma^{n_{l+1}}\omega,N-n_{l+1}}\phi_{\sigma^{n_{l+1}}\omega}\right)}
- \frac{\mathrm{Leb}\left(\widetilde G_{\omega, N} \widehat X_{\sigma^{n_{l+1}}\omega,N-n_{l+1}}
        \,\widetilde{\mathcal L}_\omega^{\,n_{l+1}}(\psi_\omega)\right)}
       {\mathrm{Leb}\left(\widehat X_{\sigma^{n_{l+1}}\omega,N-n_{l+1}} 
        \widetilde{\mathcal L}_\omega^{\,n_{l+1}}(\psi_\omega)\right)}
\right|
\Biggr)
\\
&\leq \zeta_{\omega, \infty}(J)\Bigg(
\left|\frac{\mathrm{Leb}\left(\widehat X_{\sigma^{n_{l+1}}\omega,N-n_{l+1}}
      \widetilde G_{\omega, N}(\phi_{\sigma^{n_{l+1}}\omega} - \widetilde{\mathcal L}_\omega^{\,n_{l+1}}(\psi_\omega))\right)}
     {\mathrm{Leb}\left(\widehat X_{\sigma^{n_{l+1}}\omega,N-n_{l+1}}\phi_{\sigma^{n_{l+1}}\omega}\right)}\right|
\\
&\qquad+ \frac{||g||_{\rm Lip}}
       {\mathrm{Leb}\left(\widehat X_{\sigma^{n_{l+1}}\omega,N-n_{l+1}}\phi_{\sigma^{n_{l+1}}\omega}\right)}
   \,\mathrm{Leb}\left(\left|\widehat X_{\sigma^{n_{l+1}}\omega,N-n_{l+1}}
        (\phi_{\sigma^{n_{l+1}}\omega} - \widetilde{\mathcal L}_\omega^{\,n_{l+1}}(\psi_\omega))\right|\right)
\Bigg)
\\
&\leq 2C_\phi \zeta_{\omega, \infty}(J)\,
\|\widetilde{\mathcal L}_\omega^{\,n_{l+1}}(\phi_\omega - \psi_\omega)\|_{\infty}
\|g\|_{\mathrm{Lip}}.
\end{align*}
Therefore,
since
$
\widetilde{\mathcal L}_\omega^{\,n_{l+1}}(\phi_\omega-\psi_\omega)
=
\phi_{\sigma^{n_{l+1}}\omega}
-
\widetilde{\mathcal L}_\omega^{\,n_{l+1}}(\psi_\omega)$ by Theorem~\ref{lem0916}(2) and $\nu_{\omega,\infty}(\psi_\omega)=1$,
it follows from Condition 4 applied with $u=\psi_\omega$ that 
\begin{align}
\label{ine5}
\Bigg|\frac{L_J}{L_0} - (II)_J \Bigg|
\le
C_{1}(\omega)\|g\|_{\mathrm{Lip}}\zeta_{\omega,\infty}(J)
\kappa^{n_{l+1}},
\end{align}
where
\begin{align}\label{eq:0313b1}
C_{1}(\omega)
:=
2C_\phi D\|\psi_\omega\|_{BV}.
\end{align}

In conclusion, by \eqref{ine4} and \eqref{ine5}, the left-hand side of \eqref{eq:*1} is bounded by
\begin{align}\label{eq:0313b}
&\notag\sum_{J\in\mathcal Q_\omega}\left(\left|(I)_J - \frac{L_J+R_J}{L_0+R_0}\right|+\Bigg|(II)_J - \frac{L_J}{L_0}\Bigg|+\left| \frac{L_J+R_J}{L_0+R_0}-\frac{L_J}{L_0}\right|\right)\\
&\le C_{1}(\omega)
\kappa^{n_{l+1}} + \frac{1}{|L_0+R_0|}\sum_{J \in \mathcal Q_\omega}|R_J|+\frac{|R_0|}{|L_0+R_0||L_0|}\sum_{J \in \mathcal Q_\omega} |L_J|.
\end{align}
In the following sections, we will give the desired upper bounds on $\frac{1}{|L_0+R_0|}\sum_{J \in \mathcal Q_\omega}|R_J|$ and $\frac{|R_0|}{|L_0+R_0||L_0|}\sum_{J \in \mathcal Q_\omega} |L_J|$.

\subsection{Estimates of $L_0,R_0,L_J,R_J$.}

We first estimate $|L_0|,|R_0|,\sum_{J\in\mathcal Q_\omega}|L_J|$ as follows.
\begin{lem}
\label{proesti2}
We have\begin{align*}
|R_0| &\leq D\kappa^{n_{l+1}-n_{\ast}}\lambda_\omega^{n_\ast}
\|\psi_\omega\|_{BV}\Leb(X_{\sigma^{n_{l+1}}\omega,N-n_{l+1}}), \\
|L_0| &\geq C_\phi^{-1}\Leb(X_{\sigma^{n_{l+1}}\omega,N-n_{l+1}})
\lambda_\omega^{n_\ast}
\end{align*}
and
\[
\sum_{J \in \mathcal Q_{\omega}} |L_J|
   \leq C_\phi\|g\|_{\mathrm{Lip}}
      \Leb(X_{\sigma^{n_{l+1}}\omega,N-n_{l+1}})
      \lambda_\omega^{n_\ast}.
\]
\end{lem}

\begin{proof}
For the first inequality, note that,
for any $\omega\in\Omega$, any $u\in \BV(I)$, and any integers
$m_1,m_2\ge 0$, 
\begin{align*}
Q_{\sigma^{m_1}\omega,m_2}\bigl(\mathcal L_\omega^{m_1}u\bigr)
&=
(\lambda_{\sigma^{m_1}\omega}^{m_2})^{-1}
\mathcal L_{\sigma^{m_1}\omega}^{m_2}\bigl(\mathcal L_\omega^{m_1}u\bigr)
-
\nu_{\sigma^{m_1}\omega,\infty}\bigl(\mathcal L_\omega^{m_1}u\bigr)
\phi_{\sigma^{m_1+m_2}\omega}\notag\\
&=\lambda_\omega^{m_1}(\lambda_\omega^{m_1+m_2})^{-1}\mathcal L_\omega^{m_1+m_2}u -
\lambda_\omega^{m_1}\nu_{\omega,\infty}(u)\phi_{\sigma^{m_1+m_2}\omega}\notag\\
&=\lambda_\omega^{m_1}Q_{\omega,m_1+m_2}(u).
\end{align*}
from the definition of $Q_{\omega,n}$.
Thus, by Condition 4,
\begin{align*}
|R_0|
&\le
\Bigl\|
Q_{\sigma^{n_{\ast}}\omega,n_{l+1}-n_{\ast}}
\bigl(\mathcal L_\omega^{n_{\ast}}(\psi_\omega)\bigr)
\Bigr\|_{\infty}
\Leb(X_{\sigma^{n_{l+1}}\omega,N-n_{l+1}})\\
&=
\lambda_\omega^{n_{\ast}}
\Bigl\|
Q_{\omega,n_{l+1}}(\psi_\omega)
\Bigr\|_{\infty}
\Leb(X_{\sigma^{n_{l+1}}\omega,N-n_{l+1}})\\
&\le
D\kappa^{\,n_{l+1}}\lambda_\omega^{n_{\ast}}\|\psi_\omega\|_{BV}
\Leb(X_{\sigma^{n_{l+1}}\omega,N-n_{l+1}}).
\end{align*}
Since $\kappa\in (0,1)$, this implies the first inequality.

For the second one, it follows from Theorem~\ref{lem0916}(1),
\begin{align*}
|L_0| &= \bigl|\Leb(\widehat X_{\sigma^{n_{l+1}}\omega,N-n_{l+1}}\phi_{\sigma^{n_{l+1}}\omega})
   \nu_{\sigma^{n_{\ast}}\omega, \infty}(\mathcal L_\omega^{n_\ast}(\psi_\omega))\bigr|
   \\
   &\geq C_\phi^{-1}\Leb(X_{\sigma^{n_{l+1}}\omega,N-n_{l+1}})
   \lambda_\omega^{n_{\ast}}.
\end{align*}

For the last inequality, since $\|\widetilde G_{\omega, N}\|_\infty\le \|g\|_{\mathrm{Lip}}$, for each $J \in \mathcal Q_{\omega}$ we have
\begin{align*}
|L_J| &= |\Leb\left(\widetilde G_{\omega, N} \widehat X_{\sigma^{n_{l+1}}\omega,N-n_{l+1}}\phi_{\sigma^{n_{l+1}}\omega}\right)
      \nu_{\sigma^{n_{\ast}}\omega, \infty}\left(\mathcal L_\omega^{n_\ast}(\psi_\omega \widehat{J})\right)|
\\
&\leq C_\phi\|g\|_{\mathrm{Lip}}\Leb(X_{\sigma^{n_{l+1}}\omega,N-n_{l+1}})
\lambda_\omega^{n_{\ast}}\nu_{\omega, \infty}(\psi_\omega \widehat J).
\end{align*}
Hence, we get the conclusion. 
\end{proof}

We prepare some notation for the estimate on $\sum_{J\in\mathcal Q_\omega}|R_J|$.
Recall $C_\theta$ and $n_0$ from \eqref{eq:Ctheta}.
We introduce
\[
\theta_0:=C_\theta\theta^{n_0},\qquad
\Theta:=\esssup_{\omega\in\Omega}\frac{\theta_0}{\lambda_\omega^{n_0}}<1,\qquad
\lambda_-:=\essinf_{\omega\in\Omega}\lambda_\omega,
\]
Notice that, since $\lambda_\omega =\eta_\omega(X_{\omega,1})\in [0,1]$ 
for $m$-a.e.~$\omega\in\Omega$ (cf.~\cite{AFGV}), we have
$\lambda_\omega\ge \lambda_\omega^{n_0}$  for $m$-a.e.~$\omega\in\Omega$. Thus
$\lambda_- \ge  \frac{\theta_0}{\Theta}>0.$
We also introduce
\begin{align*}
K_{\omega, n}:=8(2\xi_{\om}^{(n)}+1)\norm{((T_{\omega}^{n})^{\prime})^{-1}}_{\infty}
\dl_{\om,n}^{-1},\qquad K_n:=\esssup_{\omega\in \Omega}K_{\omega, n}.
\end{align*}
Notice that $K_n$ is finite 
by Condition 3. Hence so is
\begin{equation*}
    C_2'
:=
\max\Big\{1,C_\theta,\max_{0\le m\le n_0-1} K_m\theta_0^{-m/n_0}, K_{n_0}(1-\Theta)^{-1}\lambda_-^{-n_0}
\Big\}.
\end{equation*}
\begin{pro}
\label{proesti}
It holds that
\begin{equation*}
\sum_{J \in \mathcal Q_{\omega}} |R_J|
   \leq (6C_2'+C_\phi^2)\|\psi_\omega\|_{\mathrm{BV}}
      D\kappa^{n_{l+1}-n_{\ast}}\lambda_\omega^{n_\ast}
      \|g\|_{\mathrm{Lip}}\Leb(X_{\sigma^{n_{l+1}}\omega,N-n_{l+1}}).
      \end{equation*}
\end{pro}

In order to prove Proposition \ref{proesti}, we need the following result on Lasota--Yorke type inequality.
The Lasota--Yorke estimate is obtained by first splitting the variation over the $n$-step refined partition, i.e. $\Var(\mathcal L_\omega^n f)\le \sum_{J\in\widehat{\mathcal P}^{(n)}_\omega}\Var(\mathcal L_\omega^n(\widehat J f))$, and next by estimating each term according to the partition-level control in \cite[Lemma 1.5.1]{AFGV}.
\begin{lem}
\label{LYine}	
For all $\omega\in\Omega$, all $u\in \mathrm{BV}(I)$, and all $n\in\NN$ there exist positive, measurable constants $A_{\om}^{(n)}$ and $K_{\om}^{(n)}$ such that
\begin{align*}
\sum_{J\in \widehat{\mathcal P}^{(n)}_\omega}\var(\cL_{\om}^n(\widehat Ju)) \leq A_{\om}^{(n)}\var(u)+K_{\omega, n}\nu_{\omega, \infty}(|u|), 
\end{align*}
where
\begin{align*}
A_{\om}^{(n)}:=(9+16\xi_{\om}^{(n)})\norm{((T_{\omega}^{n})^{\prime})^{-1}}_{\infty}.
\end{align*}\index{$A_{\om}^{(n)}$}

\end{lem}

By Condition 3, we immediately obtain the following:
For all $\omega\in\Omega$, all $u\in \mathrm{BV}(I)$, and all $n\in\NN$, 
\begin{align}\label{LYine2}
\sum_{J\in \widehat{\mathcal P}^{(n)}_\omega}\var(\cL_{\om}^n(\widehat Ju))\leq C_{\theta} \theta^n \var(u)+K_{n}\nu_{\omega, \infty}(|u|).
\end{align}

\begin{proof}[Proof of Proposition \ref{proesti}]
We first note the relation that
\[
\mathcal{L}_{\sigma^{n_0}\omega}^{n_0} \bigl(\widehat J_1 \mathcal{L}_\omega^{n_0}(\psi_\omega \widehat J_2) \bigr)
= \mathcal{L}_\omega^{2n_0} \Bigl(\widehat{T_\omega^{-n_0}J_1 \cap J_2}\psi_\omega \Bigr),
\qquad J_1 \in \widehat{\mathcal{P}}_{\sigma^{n_0}\omega}^{(n_0)}, \; J_2 \in \widehat{\mathcal{P}}_\omega^{(m)}, \;
m \in \mathbb{N}.
\]
Inductively, we have that
\begin{equation*}
\mathcal{L}_{\omega}^{n_\ast}(\psi_\omega \widehat J)
= \mathcal{L}_{\sigma^{(k - 1)n_0+m}\omega}^{n_0}\Bigl(
\widehat J_1 \mathcal{L}_{\sigma^{(k-2)n_0+m}\omega}^{n_0}\bigl(
\widehat J_2 \cdots \mathcal{L}_{\sigma^{m}\omega}^{n_0}(\widehat J_k\,
\mathcal{L}_{\omega}^{m}(\widehat{\tilde{J}} \psi_\omega))\bigr)\Bigr)
\end{equation*}
for $J \in \mathcal Q_\omega$ and associated $J_1,\dots,J_k,\tilde{J}$.
We denote the right-hand side by $\mathcal{R}_{J_1,\dots,J_k,\tilde{J},\omega}(\psi_\omega)$, that is,
\begin{equation*}
\mathcal{R}_{J_1,\dots,J_k,\tilde{J},\omega}(u):= \mathcal{L}_{\sigma^{(k - 1)n_0+m}\omega}^{n_0}\Bigl(
\widehat J_1 \mathcal{L}_{\sigma^{(k-2)n_0+m}\omega}^{n_0}\bigl(
\widehat J_2 \cdots \mathcal{L}_{\sigma^{m}\omega}^{n_0}(\widehat J_k\,
\mathcal{L}_{\omega}^{m}(\widehat{\tilde{J}} u))\bigr)\Bigr).
\end{equation*}
Since
\begin{align*}
|R_J|
&\le \Leb\left(|\widetilde{G}_{\omega, N}|\,\widehat X_{\sigma^{n_{l+1}}\omega,N-n_{l+1}}\,
\left|Q_{{\sigma^{n_{\ast}}\omega},{n_{l+1}-n_{\ast}}}\mathcal{L}_\omega^{n_{\ast}}(\psi_\omega \widehat J)\right|\right)
\\
&\leq \bigl\|Q_{{\sigma^{n_{\ast}}\omega},{n_{l+1}-n_{\ast}}}(\mathcal{L}_\omega^{n_{\ast}}(\psi_\omega \widehat J))\bigr\|_{\infty}
\Leb(|\widetilde{G}_{\omega, N}|\,\widehat X_{\sigma^{n_{l+1}}\omega,N-n_{l+1}})
\end{align*}
we have, by Condition~4,
\begin{align*}
&\sum_{J \in \mathcal Q_\omega} |R_J|
\leq \sum_{J \in \mathcal Q_\omega}
\bigl\|Q_{{\sigma^{n_{\ast}}\omega},{n_{l+1}-n_{\ast}}}(\mathcal{L}_\omega^{n_{\ast}}(\psi_\omega \widehat J))\bigr\|_{\infty}
\Leb(|\widetilde{G}_{\omega, N}|\,\widehat X_{\sigma^{n_{l+1}}\omega,N-n_{l+1}})
\\
&\leq D\kappa ^{\,n_{l+1}-n_\ast}
\sum_{J \in \mathcal Q_\omega} \|
\mathcal{L}_\omega^{n_\ast}(\psi_\omega \widehat J)\|_{\mathrm{BV}}
\|g\|_{\mathrm{Lip}}\,\Leb(X_{\sigma^{n_{l+1}}\omega,N-n_{l+1}}).
\end{align*}
Using \eqref{eq:BV-L1}, we obtain
\begin{equation}\label{eq:0313}
\sum_{J \in \mathcal Q_\omega} |R_J|
\leq D\kappa ^{n_{l+1}-n_\ast}
\biggl(
\sum_{J \in \mathcal Q_\omega} 2\var\bigl(\mathcal{L}_\omega^{n_\ast}(\psi_\omega \widehat J)\bigr)
+ \sum_{J \in \mathcal Q_\omega} \|\mathcal{L}_\omega^{n_\ast}(\psi_\omega \widehat J)\|_1
\biggr)
\|g\|_{\mathrm{Lip}}\,\Leb(X_{\sigma^{n_{l+1}}\omega,N-n_{l+1}}).
\end{equation}

Since
$
\Theta
<1,
$
using \eqref{LYine2}, we can estimate
\begin{align*}
&\sum_{J\in\cQ_\omega}\mathrm{var}(\mathcal{L}_\omega^{n_\ast}(\psi_\omega \widehat J))\\
&=\sum_{J_1 \in \widehat{\mathcal{P}}_{\sigma^{(k-1)n_0 +m}\omega}^{(n_0)}} \cdots 
\sum_{\tilde{J} \in \widehat{\mathcal{P}}_\omega^{(m)}}
\var\left( \mathcal{R}_{J_1,\dots,J_k,\tilde{J},\omega}(\psi_\omega) \right)
\\
&\leq 
\sum_{J_2 \in \widehat{\mathcal{P}}_{\sigma^{(k-2)n_0 +m}\omega}^{(n_0)}} \cdots 
\sum_{\tilde{J} \in \widehat{\mathcal{P}}_\omega^{(m)}}
\theta_{0}\var\left(
\mathcal{L}_{\sigma^{(k-2)n_0+m}\omega}^{n_0}\bigl(
\widehat J_2 \cdots \mathcal{L}_{\sigma^{m}\omega}^{n_0}(\widehat J_k\,
\mathcal{L}_{\omega}^{m}(\widehat{\tilde{J}} \psi_\omega))\bigr)
\right) \\
&\qquad+ K_{n_0}\nu_{\sigma^{(k-1)n_0 +m}\omega, \infty}(\mathcal{L}_\omega^{n_{\ast}-n_0}\psi_\omega)
\\
&\leq 
\sum_{J_3 \in \widehat{\mathcal{P}}_{\sigma^{(k-3)n_0 +m}\omega}^{(n_0)}} \cdots 
\sum_{\tilde{J} \in \widehat{\mathcal{P}}_\omega^{(m)}}
\theta_{0}^2\var\left(
\mathcal{L}_{\sigma^{(k-3)n_0+m}\omega}^{n_0}\bigl(
\widehat J_3 \cdots \mathcal{L}_{\sigma^{m}\omega}^{n_0}(\widehat J_k\,
\mathcal{L}_{\omega}^{m}(\widehat{\tilde{J}} \psi_\omega))\bigr)
\right) \\
&\qquad + 
\theta_0 K_{n_0}\,\nu_{\sigma^{(k-2)n_0 +m}\omega, \infty}(\mathcal{L}_\omega^{n_{\ast}-2n_0}\psi_\omega)
+ K_{n_0}\lambda^{n_{\ast}-n_0}_{\omega}
\\
&\leq \theta_{0}^k \var\bigl(\mathcal{L}_\omega^m \psi_\omega\bigr)
+ K_{n_0}\Bigl(
\lambda_\omega^{n_\ast-n_0} + \theta_{0}\lambda_\omega^{n_\ast-2n_0}
+ \cdots + \theta_{0}^{k-1}\lambda_\omega^{n_\ast-kn_0}\Bigr),
\\
&\leq \theta_{0}^k \var\bigl(\mathcal{L}_\omega^m \psi_\omega\bigr)
+ K_{n_0}\lambda_\omega^{\,n_\ast-n_0}
\sum_{i=0}^{k-1}\Theta^i.
\end{align*}
By \eqref{LYine2} and the fact that $\Theta
<1$, 
this implies that
\begin{align*}
\notag\sum_{J\in\cQ_\omega}\mathrm{var}(\mathcal{L}_\omega^{n_\ast}(\psi_\omega \widehat J))\notag
\le\theta_{0}^k
\bigl(C_\theta \theta^{m} \var(\psi_\omega) + K_{m} \bigr)
+ K_{n_0}\lambda_\omega^{n_\ast-n_0}(1-\Theta)^{-1}.
\end{align*}
On the other hand, since 
$(\theta_0^{1/n_0})^{n_\ast}
=\theta_0^k\theta_0^{m/n_0}
\le
\lambda_\omega^{kn_0}\lambda_{\sigma^{kn_0}\omega}^{m}
=
\lambda_\omega^{n_\ast}
$
we have
\[
\theta_0^k C_\theta\theta^m
=
\bigl(C_\theta\theta^m\theta_0^{-m/n_0}\bigr)(\theta_0^{1/n_0})^{n_\ast}
=
C_\theta^{\,1-m/n_0}(\theta_0^{1/n_0})^{n_\ast}
\le
\max\{1,C_\theta\}
\lambda_\omega^{n_\ast}
\]
and
\[
\theta_0^kK_m
=
K_m\theta_0^{-m/n_0}(\theta_0^{1/n_0})^{n_\ast}
\le K_m\theta_0^{-m/n_0}\lambda_\omega^{n_\ast}.
\]
Therefore, 
\begin{align}\label{eq:Lest0116}
\notag
&\sum_{J\in\cQ_\omega}\mathrm{var}(\mathcal{L}_\omega^{n_\ast}(\psi_\omega \widehat J))\notag\\
&\leq 
C_2'
\lambda_\omega^{n_\ast}
(\var(\psi_\omega)+1)
+ K_{n_0}(1-\Theta)^{-1}
\lambda_-^{-n_0}\lambda_\omega^{n_\ast}
\notag\\
&\leq
3C_2'\|\psi_\omega\|_{\mathrm{BV}}\lambda_\omega^{n_\ast}.
\end{align}
Here we used $1=\nu_{\omega, \infty}(\psi_\omega)\le \|\psi_\omega\|_{\mathrm{BV}}$.
Furthermore,
\begin{align}\label{eq:Lest0116b}
&\sum_{J \in \mathcal Q_\omega} \notag
\bigl\|\mathcal{L}_\omega^{\,n_\ast}(\psi_\omega \widehat J)\bigr\|_1
\leq \sum_{J \in \mathcal Q_\omega}  \|\psi_\omega\| _\infty
\Leb\bigl(X_{\omega,n_\ast-1}\cap J\bigr)
\\
&
\leq \|\psi_\omega\|_{BV} C_\phi 
\Leb\bigl(\phi_\omega \widehat{X}_{\omega,n_\ast-1}\bigr)
\le \lambda_-^{-1}C_\phi^2
\|\psi_\omega\|_{BV}\lambda_\omega^{\,n_\ast},
\end{align}
where we used $\Leb(\phi_\omega \widehat{X}_{\omega,m})/\Leb(\phi_\omega \widehat{X}_{\omega,0})=\eta_\omega(X_{\omega,m})=\lambda_\omega^{m}$ (cf.~\cite{AFGV}).
The conclusion immediately follows from \eqref{eq:0313}, \eqref{eq:Lest0116} and \eqref{eq:Lest0116b}.
\end{proof}

\subsection{Final estimates for \eqref{eq:*1}}\label{s:fcb-final-estimates}
We now complete the proof of \eqref{eq:*1}.
Set
\begin{equation}\label{eq:0313b2}
C_2(\omega):=
2C_\phi D\bigl(6C_2'+C_\phi^2\bigr)\|\psi_\omega\|_{\mathrm{BV}},
\qquad
C_3(\omega):=
2C_\phi^3D
\|\psi_\omega\|_{\mathrm{BV}},
\end{equation}
and
\begin{equation*}
m_\omega:=\inf\left\{n\in\mathbb N :  C_\phi D\|\psi_\omega\|_{\mathrm{BV}}\kappa^n 
\le \frac12\right\}.
\end{equation*}

Assume first that \(n_{l+1}-n_\ast\ge m_\omega\). Then, by Lemma \ref{proesti2},
\begin{align*}
\frac{|R_0|}{|L_0|}
\le
\frac{
D\kappa^{n_{l+1}-n_\ast}\lambda_\omega^{n_\ast}\|\psi_\omega\|_{\mathrm{BV}}
\Leb(X_{\sigma^{n_{l+1}}\omega,N-n_{l+1}})
}{
C_\phi^{-1}\Leb(X_{\sigma^{n_{l+1}}\omega,N-n_{l+1}})
\lambda_\omega^{n_\ast}
}
=
C_\phi D\|\psi_\omega\|_{\mathrm{BV}}\kappa^{n_{l+1}-n_\ast}
\le \frac12.
\end{align*}
Hence,
$
|L_0+R_0|\ge |L_0|-|R_0|\ge \frac{|L_0|}{2}.
$
Therefore, it follows from Lemma \ref{proesti2} and Proposition \ref{proesti} that
\begin{align*}
\frac{1}{|L_0+R_0|}\sum_{J\in\mathcal Q_\omega}|R_J|
&\le
\frac{2}{|L_0|}\sum_{J\in\mathcal Q_\omega}|R_J|\\
&\le
\frac{2
\bigl(6C_2'+C_\phi^2\bigr)\|\psi_\omega\|_{\mathrm{BV}}
D\kappa^{\,n_{l+1}-n_\ast}\lambda_\omega^{n_\ast}\|g\|_{\mathrm{Lip}}
\Leb(X_{\sigma^{n_{l+1}}\omega,N-n_{l+1}})
}{
C_\phi^{-1}
\Leb(X_{\sigma^{n_{l+1}}\omega,N-n_{l+1}})
\lambda_\omega^{n_\ast}
}\\
&=
C_2(\omega)
\|g\|_{\mathrm{Lip}}\kappa^{\,n_{l+1}-n_\ast},
\end{align*}
and that
\begin{align*}
&\frac{|R_0|}{|L_0+R_0||L_0|}\sum_{J\in\mathcal Q_\omega}|L_J|\\
&\le
\frac{2|R_0|}{|L_0|^2}\sum_{J\in\mathcal Q_\omega}|L_J|\\
&\le
\frac{2
D\kappa^{\,n_{l+1}-n_\ast}\lambda_\omega^{n_\ast}\|\psi_\omega\|_{\mathrm{BV}}
\Leb(X_{\sigma^{n_{l+1}}\omega,N-n_{l+1}})
}{
\bigl(
C_\phi^{-1}
\Leb(X_{\sigma^{n_{l+1}}\omega,N-n_{l+1}})
\lambda_\omega^{n_\ast}
\bigr)^2
}
C_\phi\|g\|_{\mathrm{Lip}}
\Leb(X_{\sigma^{n_{l+1}}\omega,N-n_{l+1}})
\lambda_\omega^{n_\ast}\\
&=
C_3(\omega)
\|g\|_{\mathrm{Lip}}\kappa^{\,n_{l+1}-n_\ast}.
\end{align*}
Hence, since $\kappa\in (0,1)$, 
it follows from \eqref{eq:0313b} that the left-hand side of \eqref{eq:*1} is bounded by
\begin{equation*}
(C_1(\omega)+C_2(\omega)+C_3(\omega))
\|g\|_{\mathrm{Lip}}\kappa^{n_{l+1}-n_\ast}.
\end{equation*}

We next consider the case \(n_{l+1}-n_\ast<m_\omega\).
Since
\[
\sum_{J\in\mathcal Q_\omega}(I)_J
=
\sum_{J\in\mathcal Q_\omega}\int_J G_{\omega,N}(c_J,x)\,d\zeta_{\omega,N}(x),
\quad
\sum_{J\in\mathcal Q_\omega}(II)_J
=
\sum_{J\in\mathcal Q_\omega}\zeta_{\omega,\infty}(J)\int G_{\omega,N}(c_J,y)\,d\zeta_{\omega,N}(y),
\]
the left-hand side of \eqref{eq:*1} is bounded by
\[
2\|g\|_\infty\le 2\|g\|_{\mathrm{Lip}}\le
2\kappa^{-m_\omega}\|g\|_{\mathrm{Lip}}\kappa^{\,n_{l+1}-n_\ast}.
\]
We now define
\begin{equation}\label{eq:0313b3}
C_4(\omega):=2\kappa^{-m_\omega}.
\end{equation}
Then we immediately conclude \eqref{eq:*1}.

\subsection{Proof of \eqref{eq:*2}}\label{s:fcb-proof-eq2}

We first note that the left-hand side of \eqref{eq:*2} is bounded by
\[
\|g\|_{\mathrm{Lip}} \sum_{J\in\cQ_\omega} \left|\zeta_{\omega,N}(J) - \zeta_{\omega,\infty}(J)\right|.
\]
Thus it suffices to estimate
\[
\sum_{J\in\cQ_\omega} \left|\zeta_{\omega,N}(J) - \zeta_{\omega,\infty}(J)\right|.
\]

Fix \(J\in\cQ_\omega\).
Since
\[
\Leb\big(\widehat X_{\sigma^N\omega, 0}\tcL_\omega^{\,N}(u)\big)
=
\Leb\big(u\widehat X_{\omega,N}\big)
\]
for any \(u\in \BV(I)\), we obtain
\[
\zeta_{\omega,N}(J)
=
\frac{\Leb\big(\widehat X_{\sigma^N\omega,0}\tcL_\omega^{\,N}(\widehat J\psi_\omega)\big)}
{\Leb\big(\widehat X_{\sigma^N\omega,0}\tcL_\omega^{\,N}(\psi_\omega)\big)}.
\]
According to the spectral decomposition
$
\tcL_\omega^N(\psi_\omega)=\phi_{\sigma^N\omega}+Q_{\omega,N}(\psi_\omega)
$ in Theorem \ref{lem0916} (recall that $\nu_{\omega,\infty}(\psi_\omega)=1$),
we set
\begin{align*}
&A_0 := \Leb(\widehat X_{\sigma^N\omega,0}\phi_{\sigma^N\omega}),\quad
B_0 := \Leb(\widehat X_{\sigma^N\omega,0}Q_{\omega,N}(\psi_\omega)),\\
&A_J := \nu_{\omega,\infty}(\widehat J\psi_\omega)=\zeta_{\omega,\infty}(J),\quad
B_J := \Leb(\widehat X_{\sigma^N\omega,0}Q_{\omega,N}(\widehat J\psi_\omega)).
\end{align*}
Then
\[
|\zeta_{\omega,N}(J)-\zeta_{\omega,\infty}(J)|
=
\left|\frac{A_JA_0+B_J}{A_0+B_0}-A_J\right|
\le
\frac{1}{|A_0+B_0|}|B_J|
+\frac{|B_0|}{|A_0+B_0|}|A_J|.
\]

By Condition~4,
\begin{align}\label{eq:s54-B0}
\begin{split}
&|A_0|\ge C_\phi^{-1},\\
&|B_0|
\le
\|Q_{\omega,N}(\psi_\omega)\|_\infty
\le
D\|\psi_\omega\|_{\BV}\kappa^N.
\end{split}
\end{align}
Notice also that
\[
\sum_{J\in\cQ_\omega}|A_J|
=
\sum_{J\in\cQ_\omega}\nu_{\omega,\infty}(\widehat J\psi_\omega)
=
\nu_{\omega,\infty}(\psi_\omega)
=
1.
\]
On the other hand,
since
\[
Q_{\omega,N}
=
\lambda_\omega^{-n_\ast}Q_{\sigma^{n_\ast}\omega,N-n_\ast}\circ \mathcal L_\omega^{\,n_\ast},
\]
Condition~4 implies
\begin{align*}
\sum_{J\in\cQ_\omega}|B_J|
&\le
\sum_{J\in\cQ_\omega}\|Q_{\omega,N}(\widehat J\psi_\omega)\|_\infty\\
&\le
D\kappa^{N-n_\ast}\lambda_\omega^{-n_\ast}
\sum_{J\in\cQ_\omega}\bigl\|\cL_\omega^{\,n_\ast}(\widehat J\psi_\omega)\bigr\|_{\BV}\\
&\le
D\kappa^{N-n_\ast}\lambda_\omega^{-n_\ast}
\left(
2\sum_{J\in\cQ_\omega}\var(\cL_\omega^{n_\ast}(\psi_\omega \widehat J))
+
\sum_{J\in\cQ_\omega}\bigl\|\cL_\omega^{\,n_\ast}(\psi_\omega \widehat J)\bigr\|_1
\right),
\end{align*}
where we used \eqref{eq:BV-L1}.
By \eqref{eq:Lest0116}, \eqref{eq:Lest0116b} and the fact that $\lambda_\omega \in [0,1]$, this yields
\begin{align*}
\sum_{J\in\cQ_\omega}|B_J|
&\le
D\bigl(
6C_2'
+
\lambda_-^{-1}C_\phi^2
\bigr)\|\psi_\omega\|_{\BV}\kappa^{N-n_\ast}.
\end{align*}

We first consider the case \(N\ge m_\omega\).
Then \(|B_0|\le A_0/2\), so \(|A_0+B_0|\ge A_0/2\).
Hence, by the above estimates, the left-hand side of \eqref{eq:*2} is bounded by
\[
2\left(\frac{1}{A_0} \sum_{J\in\cQ_\omega}|B_J|+\frac{|B_0|}{A_0}\sum_{J\in\cQ_\omega}|A_J|\right)\|g\|_{\mathrm{Lip}}
\le C_5(\omega)\|g\|_{\mathrm{Lip}}\kappa^{N-n_\ast}\le C_5(\omega)\|g\|_{\mathrm{Lip}}\kappa^{n_{l+1}-n_\ast},
\]
where
\begin{equation}\label{eq:0314}
C_5(\omega):= 2C_\phi D\bigl(6C_2'+\lambda_-^{-1}C_\phi^2+1\bigr)\|\psi_\omega\|_{\BV}.
\end{equation}

We next consider the case \(N<m_\omega\).
Since \(\zeta_{\omega,N}\) and \(\zeta_{\omega,\infty}\) are probability measures, we have
\[
\sum_{J\in\cQ_\omega}\left|\zeta_{\omega,N}(J)-\zeta_{\omega,\infty}(J)\right|\le 2
\le
2\kappa ^{-m_\omega}\kappa ^{N}
\le
2\kappa ^{-m_\omega}\kappa ^{n_{l+1}-n_*}.
\]
Hence \eqref{eq:*2} immediately follows.

\section{Proof of Theorem \ref{CLT}}\label{s:abst}

In this section, we formulate an abstract normal approximation theorem,
Theorem \ref{thm:clt_tikhomirov}, which we use to prove Theorem \ref{CLT}.
Apart from the estimate on the Kolmogorov distance, proved in
Section \ref{s:5.4}, the result is a  straightforward consequence of
\cite{LS20,LNN25}. 
We first introduce in Section \ref{s:abstr} the definitions
required to state Theorem \ref{thm:clt_tikhomirov}, and then present the
theorem itself in Section \ref{s:5.2}. 
Next, after a preparation in Section \ref{ss:cl}, we provide an estimate on the dynamical
variance in Section \ref{s:5.3}, which together with Theorem \ref{FCB} is used to deduce the error
bounds stated in Theorem \ref{CLT}.

\subsection{Barbour's class of smooth test functions}\label{s:abstr}
Recall that $D$ is the space of all c\`{a}dl\`{a}g functions $w : [0,1] \to \mathbb R$ equipped with the sup norm $\Vert w \Vert_\infty = \sup_{ t \in [0,1] } | w(t)|$.
Given a function $g : D \to \mathbb R$, by $g^{(k)}$ we mean the $k$-th Fr\'{e}chet derivative of $g$, which is 
a map 
$g^{(k)} : D \to \cL( D^k , \mathbb R )$ from $D$ to the space
$\cL( D^k, \mathbb R )$ of all continuous 
multilinear maps from $D^k$ to $\mathbb R$. The $k$-linear norm of $A \in \cL( D^k, \mathbb R )$
is defined by
$$
\Vert A \Vert = \sup_{ \Vert w_i \Vert_\infty \le 1 \: \forall i = 1,\ldots, k } | A[w_1,\ldots, w_k] |,
$$
where 
$A[w_1,\ldots, w_k]$ denotes $A$ applied to the arguments $w_1,\ldots, w_k \in D$.
Let $\sL$ be the Banach space of all continuous functions $g : D \to \mathbb R$ for which the norm  
$$
\Vert g \Vert_{\sL} := \sup_{w \in D} \frac{ | g(w) | }{  1 + \Vert w \Vert_\infty^3 }
$$
is finite. 
For a twice Fr\'{e}chet differentiable functions $g:D\to\mathbb R$, we define
\begin{align}\label{eq:norm_m0}
	\begin{split}
		\Vert g \Vert_\sM &= \sup_{w \in D} \frac{ | g(w) | }{ 1 + \Vert w \Vert_\infty^3 } +  
		\sup_{w \in D} \frac{  \Vert  g' (w) \Vert   }{ 1 + \Vert w \Vert_\infty^2 }  + 
		\sup_{w \in D} \frac{  \Vert  g'' (w) \Vert   }{ 1 + \Vert w \Vert_\infty } \\
		&+ 
		\sup_{w,h \in D}  \frac{ \Vert g''(w + h) - g''(w)  \Vert  }{ \Vert h \Vert_\infty }.
		\end{split}
	\end{align}
We write $\sM$ for the collection of all
twice Fr\'{e}chet differentiable functions
$g:D\to\mathbb R$ such that
\begin{align*}
\sup_{w,h \in D, \, h \neq 0} \frac{\Vert g''(w + h) -  g''(w)   \Vert}{  \Vert h \Vert_\infty } < \infty.
\end{align*}
Then \eqref{eq:norm_m0} defines a norm on $\sM$.
Let $\sM_0 \subset \sM$ be the subcollection of 
functions that satisfy
\begin{align}\label{eq:smooth}
	\sup_{w \in D}	| g''(w)[J_r, J_s - J_t] |   
	\le C_0 \Vert g \Vert_{\sM } |t-s|^{1/2} \quad \forall r,s,t \in [0,1],
\end{align}
where 
\begin{align}\label{eq:Jt}
	J_{ \alpha }(t) = \begin{cases}
		1, & \text{if $t \ge \alpha$}, \\ 
		0, & \text{if $t < \alpha$}
	\end{cases}.
\end{align}
We refer the reader to \cite{B90, BRZ24} for 
background information on the class 
$\sM_0$. We define
$d_{\cS}(\cdot, \cdot)$ as in \eqref{eq:smooth_dist}.

\subsection{Normal approximations under FCB}\label{s:5.2}

For $N \ge 1$, let $\{\tilde Y_n \}_{0\le n<N}$ be real-valued random variables on 
a probability space $(M, \cF, \mu)$ such that
\begin{align}\label{eq:X_j}
	\mu(\tilde Y_n) = 0 \quad \text{and} \quad  \Vert \tilde Y_n \Vert_\infty \le L \qquad \forall n = 0, \ldots, N - 1,
\end{align}
where $L \ge 1$.

\begin{df}
We say that $\{\tilde Y_n \}_{0\le n<N}$ satisfies the \emph{functional correlation bound} (FCB) with
a rate function $R : \{1,2,\ldots\} \to \RR_+$ and a constant $C_* \ge 1$ if for any $k \ge 2$, any separately Lipschitz continuous function $g : [-L,L]^k \to \RR$, and any integers $1 \le \ell < k$ and $0 \le n_1 < \cdots < n_k < N$,
\begin{align}\label{eq:fcb}
	&\left| \int   G(x,x)   \, d\mu(x) - \iint G(x,y) \, d\mu(x) \, d\mu(y) \right| \le C_*  \Vert g \Vert_{ \text{Lip}  } R( n_{\ell + 1} - n_\ell), \tag{$\text{FCB}$}
\end{align}
where
\begin{align}\label{eq:G}
	G(x,y) = g( \tilde Y_{n_1}(x), \ldots, \tilde Y_{n_\ell}(x), \tilde Y_{n_{\ell+1}}(y), \ldots, \tilde Y_{n_k}(y)  ).
\end{align}
\end{df}

\begin{rem}\label{rem:fcb} If \eqref{eq:fcb}
holds and $g : [-L, L]^k \to \CC$ is a separately Lipschitz continuous 
complex-valued
function, then for $G$ defined as in \eqref{eq:G} we have 
\begin{align*}
	&\left| \int   G(x,x)   \, d \mu (x) - \iint G(x,y) \, d \mu (x) \, d \mu (y) \right| \le 2C_*  \Vert g \Vert_{ \text{Lip}  } R( n_{\ell + 1} - n_\ell).
\end{align*}
Henceforth, we apply (FCB) also to complex-valued functions $g$ without explicitly mentioning this extension. Moreover,  suppose that $0 \le n_1 <  \ldots < n_k < N$, and $1 < \ell_1 < \cdots < \ell_p < k$ are integers, $k \ge 2$, 
and that $g : [-L, L]^k \to \CC$ is  separately Lipschitz continuous. Define
\begin{align*}
	&G(x_1, \ldots, x_{p+1}) \\
	&= g( \tilde Y_{n_1}(x_1), \ldots , \tilde Y_{n_{ \ell_1 } }(x_1),   \tilde Y_{n_{ \ell_1 + 1 } }(x_2), \ldots, 
    \tilde Y_{n_{ \ell_2 } }(x_2 ), 
	\ldots, \tilde Y_{  n_{ \ell_p + 1 } } (x_{p+1}), \ldots \tilde Y_{n_k}(x_{p+1}) ).
\end{align*}
Then, (FCB) implies the following upper bound (see \cite[Proposition 2.8]{LNN25}).
\begin{align}\label{eq:fcb_many_gaps}
	\begin{split}
		&\biggl| \int G(x,\ldots, x) \, d \mu(x) - \int\cdots\int G(x_1, \ldots, x_{p+1}) \, d \mu(x_1) \ldots   \, d \mu(x_{p+1}) \biggr| \\
		&\le 2 C_* \Vert g \Vert_{ \text{Lip} } \sum_{j=1}^p  R( n_{ \ell_{ j } + 1 } - n_{ \ell_j } ). 
	\end{split}
	\tag{$\text{FCB'}$}
\end{align}
\end{rem}

For $N \ge 1$ and $0 \le n < N$, and $t \in [0,1]$, we define
\begin{align}\label{eq:defs}
Y_n = \sigma_N^{-1}\tilde Y_n , \quad W = \sum_{n=0}^{N - 1} Y_n,\quad
	\mathscr W(t) =   \sum_{n = 0}^{  N - 1  } 
	J_{ \sigma_{N,n}^2 / \sigma_N^2 }(t)
	Y_n,
\end{align}
where 
\begin{align}\label{eq:sigmas}
\sigma_{N, n}^2 = \sum_{i=0}^{n-1} \sum_{j=0}^{N-1} \mu( \tilde Y_i \tilde Y_j ), \quad 
\sigma_N^2 = \sigma_{N,N}^2 = 
\mu\biggl[ \biggl(\sum_{n=0}^{N - 1} \tilde Y_n \biggr)^2 \biggr].
\end{align}
We assume that $\sigma_N > 0$.
Recall that $Z$ is a standard normal random variable and $\mathscr Z$ is a standard Brownian motion.

\begin{thm}\label{thm:clt_tikhomirov} Assume that $\{\tilde Y_n \}_{0\le n<N}$ satisfies \eqref{eq:fcb} with rate function $R$ and constant $C_*$.
Set $\hat{R}(n) = \sup_{\ell \ge n} R(\ell)$ for $n \ge 1$, $\hat{R}(0) = 1$, and 
	$$
	\hat{R}_1(N) = \sum_{n=0}^{N-1} \hat{R}(n), \quad \hat{R}_2(N) = \sum_{n=0}^{N-1} n^{1/2} \hat{R}(n), \quad 
	\hat{R}_3(N) = \sum_{n=0}^{N-1} n \hat{R}(n).
	$$
    Then, there exists an absolute constant $C > 0$ such that the following upper bounds hold.
\begin{itemize}
    \item[(i)] The Wasserstein distance satisfies
	\begin{align}\label{eq:ub_w1}
		d_{\mathcal W}(W,Z) \le C C_* L^5  \sigma_N^{-3} N (1 + \hat{R}_3(N)).
	\end{align}
   \item[(ii)] For any integer $K$ satisfying
	$$
	2 + \frac{1}{2} \log (N) /\log(2)  < K \le N,
	$$
	the Kolmogorov distance satisfies
	\begin{align}\label{eq:clt_tikhomirov}
	\begin{split}
		d_\cK(W, Z) 
		&\le 
        C L^2  \sigma_N^{-2} N \sum_{j=K}^{N-1} R(j) + C C_* L^3 K^2 N \sigma_N^{-3} 
	+ C C_* L N \sigma_N^{-1} \hat{R}(K) \\
	&+  C \sqrt{N} \sigma_N^{-2}  L^2 K
	+  C \sqrt{C_*} \sqrt{N} \sigma_N^{-2} L^2 K^{3/2} \biggl\{ 1 + \sqrt{\hat{R}_1(N)} \biggr\} + C LK \sigma_N^{-1}.
	\end{split}
	\end{align}
       \item[(iii)] The integral distance of smooth test functions satisfies
	\begin{align}\label{eq:error_fclt}
		d_{\cS }(\mathscr W,\mathscr Z) \le C  C_{N}  \sigma_N^{-3} N ,
	\end{align}
	where
	\begin{align*}
		C_{N} = (L+1)^3 \biggl\{  1 + ( C_*  + 1)  \hat{R}_3(N)
		+ C_0C_*^{3/2}    \hat{R}_1(N) ^{1/2}( 1 +  \hat{R}_2(N)) \biggr\}.
	\end{align*}
	Moreover, if $\{\tilde Y_n \}_{0\le n<N}$ satisfies (FCB) for all $N \ge 1$ with 
	constant $C_*$ and 
	rate function 
	$R$ such that
	\begin{align}\label{eq:rates_R}
	\sum_{n = 1}^\infty nR(n) < \infty, \quad \sum_{n=1}^\infty \sqrt{R(n)} < \infty,
	\end{align}
	and if there exists $\Sigma^2 > 0$ such that 
	\begin{align}\label{eq:conv_var}
	\lim_{N\to \infty} \frac{\sigma_N^{2}}{N} = \Sigma^2,
	\end{align}
	then $\mathscr W$ converges weakly to $\mathscr Z$ in the Skorokhod topology as $N \to \infty$.
    \end{itemize}
\end{thm}

\begin{proof} Let $C$ denote various absolute constants. For convenience, we set $R(0)=1$.
We also use the following conventions. For any finite subset $J \subset \mathbb{Z}_+$, we write $x_J$
for a generic element $(x_j)_{j \in J} \in [-L, L]^J$.
Using an analogous convention, we write 
$\tilde Y_J = (\tilde Y_j)_{j \in J}$.
For two subsets $J_1, J_2 \subset \mathbb{Z}_+$ that are ordered in the sense that 
every element of $J_1$ is smaller than every element of $J_2$, 
we write $(x_{J_1}, x_{J_2})$ for the concatenation of the corresponding 
vectors. If $J_1 = \varnothing$, we identify $(x_{J_1}, x_{J_2})$ with $x_{J_2}$, and 
more generally, whenever a concatenation involves an empty index set,
the corresponding vector is omitted.

\smallskip 
	
\noindent \textbf{(i):} The upper bound \eqref{eq:ub_w1}
follows by verifying conditions (B1)-(B3) in \cite[Theorem 2.3]{LS20}. Indeed, as noted in \cite[Remark 2.2]{LS20}, this theorem remains valid when the
functions $f^{i}$ in its statement are replaced by the general random variables
$\tilde Y_i$.

By 	\cite[Proposition 2.6]{LNN25}, for $0 \le i \le j < N$,
\begin{align*}
	| \mu( \tilde Y_i \tilde Y_j ) | \le C C_* L^2  R(j-i),
\end{align*}
This proves that (B1) in \cite[Theorem 2.3]{LS20} holds 
with $C_1 = C C_* L^2$ and $\rho(n) = R(n)$.

Next, we verify (B2). To this end, let $m \le k \le N-1$ and let $\Psi : \RR \times [ -4L - 1, 4 L + 1 ] \to \RR$ be a bounded Lipschitz continuous function. Write 
$$
\tilde Y_{n,m} = \sum_{ \substack{ 0 \le i < N \\  |i - n| = m }  } \tilde Y_i. 
$$
Note that 
\begin{align*}
\mu \biggl[  \tilde{Y}_n  \Psi  \biggl( \sum_{  \substack{ 0 \le i < N \\ |i - n| > k }  } \tilde Y_i, \tilde Y_{n,k}   \biggl)  \tilde Y_{n,m}   \biggr] = \mu  \biggl[   g(  \tilde Y_{  J_{-}  }, \tilde Y_n, \tilde Y_{ J_+ }     )   \biggr],
\end{align*}
where 
\begin{align*}
	&J_{-} = \{  0 \le j < N \: : \:   j \le n - k \} \cup \{   0 \le j < N  \: : \:  j = n - m \}, \\ 
	&J_{+}  = \{  0 \le j < N \: : \:   j \ge n + k \} \cup \{   0 \le j < N  \: : \:  j = n + m \}, \\
	&g( x_{ J_{-} } , x_n, x_{J_+ } ) =   x_n  \Psi  \biggl( \sum_{  \substack{ 0 \le i < N \\ |i - n| > k }  }  x_i , 
	\sum_{ \substack{ 0 \le i < N \\  |i - n| = k }  } x_i
	  \biggr) \sum_{ \substack{ 0 \le i < N \\  |i - n| = m }  } x_i.
\end{align*}
It is straightforward to verify that
$$
\Vert g \Vert_{\Lip} \le C L^2 \Vert \Psi \Vert_{ \Lip  }.
$$
Since $\mu(  \tilde Y_n ) = 0$, an application of \eqref{eq:fcb_many_gaps} yields
\begin{align*}
	\biggl| \mu \biggl[  \tilde{Y}_n  \Psi  \biggl( \sum_{  \substack{ 0 \le i < N \\ |i - n| > k }  } \tilde Y_i, \tilde Y_{n,k}   \biggl)  \tilde Y_{n,m}   \biggr] \biggl| \le C C_* L^2 \Vert \Psi \Vert_{\Lip} R(m).
\end{align*}
That is, (B2) in \cite[Theorem 2.3]{LS20} holds with $C_2 = C C_* L^2$. Similarly, we see that 
(B3) in \cite[Theorem 2.3]{LS20} holds with $C_3 = C C_* L^2$. Now, \eqref{eq:ub_w1} follows 
by  \cite[Theorem 2.3]{LS20}.

\smallskip

\noindent \textbf{(ii):} The proof of the claimed upper bound on $d_{\cK}(W,Z)$ is given in Section \ref{s:5.4}. 

\smallskip 

\noindent\textbf{(iii):} The upper bound \eqref{eq:error_fclt} was established in \cite[Theorem 2.11]{LNN25}, so we only need to show the last statement about weak convergence. This in turn follows from \cite[Proposition 2.14]{LNN25}, once we verify the following conditions:
\begin{itemize}
	\item[(a)] $\sigma_N^{-1} = o(N^{-2/3})$. \smallskip
	\item[(b)] There exist $A > 0$,  
	$N_0 \ge 1$ and $M_N \ge 1$ with $M_N = o(\sigma^2_N)$
	such that, whenever $N \ge N_0$,
	\begin{align*}
		\inf_{ 0 \le u,v < N, \:  v - u \ge M_N } ( \sigma_{N,v}^2 - \sigma_{N,u}^2 )  \ge A M_N.
	\end{align*}
\end{itemize}
Condition (a) holds due to the assumption \eqref{eq:conv_var}.

\smallskip 

\noindent\emph{Proof of (b).}  Define 
\begin{align*}
	M_N = \min \{ m \ge \sqrt{N} \: : \: m \ge \frac{8}{\Sigma^2} \cE(m) N  \},
\end{align*}	
where
$$
\cE( m ) = \sup_{ k \ge m } |  \Sigma^2 - k^{-1} \sigma_{k}^2  |.
$$
Then $\cE(m) \downarrow 0$ as $m \to \infty$. There exists an integer $K_1$ such that
$$
N \ge  \frac{8}{\Sigma^2} \cE(N) N
$$
for all $N \ge K_1$.
Consequently, $M_N \le N$ and $\cE(N) \le \Sigma^2 / 8$
for $N \ge K_1$.
Further, note that 
$$
M_N = o(\sigma^2_N) = o(N).
$$
Indeed, by minimality of $M_N$ we have 
\begin{align*}
	\frac{M_N}{N} \le \frac1N \biggl( \frac{8}{\Sigma^2} \cE( M_N - 1 ) N + 1  \biggr) 
    + \biggl( \frac1N + \frac{1}{\sqrt{N}} \biggr)
    \to 0, \quad \text{as $N \to \infty$.}
\end{align*}
Here, we used $\lim_{N\to \infty} M_N = \infty$.

Restrict to $N \ge K_1$ and let $0 \le u \le u + h < N$ be integers such that $h \ge M_N$.
Using \eqref{eq:rates_R} and the definition of $\cE$, we have
\begin{align*}
	\sigma_{N, u+h}^2 - \sigma_{N, u}^{2} \ge h \Sigma^2 - u \cE(u) - (u + h) \cE(u + h) - C_1 
	\ge  h \Sigma^2 - u \cE(u) - 2 N \cE(M_N) - C_1,
\end{align*}
for some constant $C_1 > 0$ independent of $N, u, h$. Moreover, for all $u \le N$,
\begin{align*}
	u \cE(u)  &\le  u \cE(u) \mathbf{1}_{  u \le K_1 } + u \cE(u) \mathbf{1}_{  K_1 < u \le M_N }
	+
	 u \cE(u) \mathbf{1}_{ u \ge M_N } \\
	&\le \sup_{w \le K_1 } w \cE(w) + M_N \frac{\Sigma^2}{8} + N \cE(M_N).
\end{align*}
It follows that 
\begin{align*}
	\sigma_{N, u+h}^2 - \sigma_{N,u}^{2}
	\ge  M_N \Sigma^2 - M_N  \frac{\Sigma^2}{2} - C_2 \ge 
	 \frac{\Sigma^2}{2} M_N   - C_2,
\end{align*}
for some constant $C_2 > 0$ independent of $N, u, h$. Here, we used the definition of $M_N$ to 
obtain the first inequality. Choose $N_0 \ge K_1$ such that $M_N \ge C_2 4 / \Sigma^2$ for all 
$N \ge N_0$. Then, for all $N \ge N_0$,
\begin{align*}
	\sigma_{N, u+h}^2 - \sigma_{N, u}^{2} \ge 
	\frac{\Sigma^2}{4} M_N.
\end{align*}
Consequently, (b) holds with $A = \Sigma^2 / 4$.

\hfill
\qed$_{(b)}$

Proposition 2.14 of \cite{LNN25} now implies that  $\mathscr W$ converges weakly to $\mathscr Z$ in the Skorokhod topology as $N \to \infty$.
This completes the proof of (iii).
\end{proof}

\begin{cor}\label{cor:kolmogorov}
Assume that $\{\tilde Y_n \}_{0\le n<N}$ satisfies \eqref{eq:fcb}
for all $N \ge 1$ with 
constant $C_*$ and rate function
$R(n) = \theta^n$, where $\theta \in (0,1)$. Then, 
	there exists a constant $C_\theta > 0$ depending only on $\theta$ such that, for $N \ge 1$,
	\begin{align*}
		d_\cK(W, Z) \le  C_\theta C_* L^3 \log^2(N + 1) (  N^{-1/2} + \sigma_N^{-1} + N \sigma_N^{-3} + \sqrt{N} \sigma_N^{-2}  ).
	\end{align*}
	In particular, if $\sigma_N^{-1} = O(N^{-1/2})$, we obtain $d_\cK(W, Z) = O ( \log^2 (N)  N^{-1/2} )$.
\end{cor}
\begin{proof}
	The desired estimate follows immediately by choosing $K = \lceil C_\theta  \log (N)  /  \log( 1 / \theta ) \rceil$ 
	with a suitably large $C_\theta>0$
	in Theorem \ref{thm:clt_tikhomirov}(2).
\end{proof}

\subsection{Conditioning limits}\label{ss:cl}
In this subsection, as a preparation for the linear growth of the dynamical variance $\sigma_{N}$ proven in the next subsection, we establish the existence of conditioning limits.
Recall (\ref{measzeta}) for the definition of $\zeta_{\omega, n}$ for a random probability measure $\zeta_{\omega}.$ 
In the special case that 
$\zeta_{\omega}=\eta_{\omega}$ we 
write
$\eta_{\omega,n}$ for $\zeta_{\omega, n}$,
which more specifically means that
\[
   \eta_{\omega,n}(A)
   = \frac{
     \eta_{\omega} \left( A \cap X_{\omega,n} \right)
   }{
     \eta_{\omega} \left( X_{\omega,n} \right)
   }
\]
for a measurable subset $A \subseteq I$.

\begin{lem}\label{expozeta}
For $m$-a.e.~$\omega\in\Omega$, for each $u,v\in \BV_\Omega(I)$ and each integer $k\ge 1$, we have
\[
\lim_{N\to\infty}
\eta_{\omega,N}\bigl(u_\omega\cdot v_{\sigma^k\omega}\circ T_\omega^k\bigr)
=
\eta_{\omega,\infty}\bigl(u_\omega\cdot v_{\sigma^k\omega}\circ T_\omega^k\bigr).
\]
\end{lem}

\begin{proof}
Fix $u,v\in \BV_\Omega(I)$, $k\ge 1$, and $\omega$ in a full-measure set on which all conclusions of Theorem~\ref{lem0916} and Condition~4 hold.
For each integer $N>k$, using
\[
\widehat X_{\omega,N}
=
\widehat X_{\omega,k}\cdot
\bigl(\widehat X_{\sigma^k\omega,N-k}\circ T_\omega^k\bigr),
\]
the duality relation \eqref{eq:0115a}, and
$
\widetilde{\cL}_\omega^k:=(\lambda_\omega^k)^{-1}\cL_\omega^k,
$
we obtain
\begin{align*}
\eta_{\omega,N}\bigl(u_\omega\cdot v_{\sigma^k\omega}\circ T_\omega^k\bigr)
&=
\frac{
\Leb\bigl(
\widehat X_{\sigma^k\omega,N-k}\,
\widetilde{\cL}_\omega^k(\phi_\omega u_\omega)\,
v_{\sigma^k\omega}
\bigr)
}{
\Leb\bigl(
\widehat X_{\sigma^k\omega,N-k}\phi_{\sigma^k\omega}
\bigr)
}.
\end{align*}
Applying the same relation once more, this becomes
\begin{align*}
\eta_{\omega,N}\bigl(u_\omega\cdot v_{\sigma^k\omega}\circ T_\omega^k\bigr)
&=
\frac{
\Leb\Bigl(
\widehat X_{\sigma^N\omega,0}\,
\widetilde{\cL}_{\sigma^k\omega}^{\,N-k}
\bigl(\widetilde{\cL}_\omega^k(\phi_\omega u_\omega)\,v_{\sigma^k\omega}\bigr)
\Bigr)
}{
\Leb\bigl(
\widehat X_{\sigma^N\omega,0}\phi_{\sigma^N\omega}
\bigr)
}.
\end{align*}
Now the decomposition 
$\widetilde{\cL}_{\sigma^k\omega}^{\,N-k}
(u)
=
\nu_{\sigma^k\omega,\infty}(u)
\phi_{\sigma^N\omega}
+
Q_{\sigma^k\omega,N-k}(u)
$
gives
\begin{align*}
\eta_{\omega,N}\bigl(u_\omega\cdot v_{\sigma^k\omega}\circ T_\omega^k\bigr)
=
\nu_{\sigma^k\omega,\infty}
\bigl(\widetilde{\cL}_\omega^k(\phi_\omega u_\omega)\,v_{\sigma^k\omega}\bigr)+
\frac{
\Leb\Bigl(
\widehat X_{\sigma^N\omega,0}\,
Q_{\sigma^k\omega,N-k}
\bigl(\widetilde{\cL}_\omega^k(\phi_\omega u_\omega)\,v_{\sigma^k\omega}\bigr)
\Bigr)
}{
\Leb\bigl(
\widehat X_{\sigma^N\omega,0}\phi_{\sigma^N\omega}
\bigr)
}.
\end{align*}
By Theorem~\ref{lem0916} and Condition~4, the second term tends to $0$ as $N\to\infty$.
Therefore
\[
\lim_{N\to\infty}
\eta_{\omega,N}\bigl(u_\omega\cdot v_{\sigma^k\omega}\circ T_\omega^k\bigr)
=
\nu_{\sigma^k\omega,\infty}
\bigl(\widetilde{\cL}_\omega^k(\phi_\omega u_\omega)\,v_{\sigma^k\omega}\bigr).
\]
Finally, by \eqref{eq:0115a},
\begin{align*}
\nu_{\sigma^k\omega,\infty}
\bigl(\widetilde{\cL}_\omega^k(\phi_\omega u_\omega)\,v_{\sigma^k\omega}\bigr)
=
\nu_{\omega,\infty}
\bigl(\phi_\omega u_\omega\cdot v_{\sigma^k\omega}\circ T_\omega^k\bigr)=
\eta_{\omega,\infty}
\bigl(u_\omega\cdot v_{\sigma^k\omega}\circ T_\omega^k\bigr),
\end{align*}
which proves the claim.
\end{proof}

\subsection{Linear growth of dynamical variance $\sigma_N^2$}\label{s:5.3}
In this subsection, we investigate the dynamical variance $\sigma_N^2$ under Conditions 1--4.

\begin{pro}\label{prop:0625a}
Let $f\in \BV_\Omega (I)$ satisfy $\eta_{\omega,\infty}(f_\omega) = 0$ for $m$-a.e.~$\omega\in\Omega$.
Then the following holds:
\begin{enumerate}
    \item
The following limit exists for $m$-a.e.~$\omega\in\Omega$:
$$ \lim_{N\to\infty}     \eta_{\omega,N}\Big[\frac{1}{N}\Big( \sum_{n=0}^{N-1} f_{\sigma^n\omega} \circ T^n_\omega\Big)^2\Big]   =  \Sigma^2(f), $$
where
$$ \Sigma^2(f) := \int_\Omega \eta_{\omega,\infty} (f^2_\omega)\, dm(\omega) + 2 \sum_{n=1}^{\infty} \int_\Omega \eta_{\omega,\infty}( f_\omega \cdot f_{\sigma^n\omega} \circ T^n_\omega)\, dm(\omega). $$
 \item $\Sigma^2(f)=0$ if and only if 
 $f$ is a coboundary, i.e., 
 there exists $\tilde f\in L^2(\eta_{\infty})$ such that $f_\omega(x)=\tilde f_\omega(x) - \tilde f_{\sigma\omega}\circ T_\omega(x) $ for $m$-a.e.~$\omega\in\Omega$ and $\eta_{\omega,\infty}$-a.e.~$x\in X_{\omega,\infty}$.
 \end{enumerate}
\end{pro}

\begin{proof}
We expand the square and split the expression into diagonal and off-diagonal terms:
\begin{align*}
&\eta_{\omega,N}\Big[\frac{1}{N}\Big( \sum_{n=0}^{N-1} f_{\sigma^n\omega} \circ T^n_\omega\Big)^2\Big] \\
&= \frac{1}{N} \sum_{n=0}^{N-1} \eta_{\omega,N} \bigl((f_{\sigma^n\omega} \circ T^n_{\omega})^2\bigr)
 + \frac{2}{N} \sum_{i=0}^{N-2} \sum_{j=i+1}^{N-1} \eta_{\omega,N}\bigl(f _{\sigma^i\omega}\circ T^i_{\omega} \cdot f_{\sigma^j\omega} \circ T^j_{\omega}\bigr).
\end{align*}

For the diagonal term, by the conditional invariance of $\eta$ (i.e.~$\eta_{\omega,1}(T_\omega^{-1}(\cdot))=\eta_{\sigma\omega,0}$ $m$-a.s.~$\omega\in\Omega$; cf.~\cite{AFGV}),
\[
\frac{1}{N} \sum_{n=0}^{N-1} \eta_{\omega,N} \bigl((f_{\sigma^n\omega} \circ T^n_{\omega})^2\bigr)
=
\frac{1}{N} \sum_{n=0}^{N-1} \eta_{\sigma^n\omega,N-n}(f^2_{\sigma^n\omega})
\]
for $m$-a.e.~$\omega\in\Omega$.
Since $\widehat I\in \BV_\Omega(I)$, Lemma~\ref{expozeta} applied to $u=f^2$ and $v= \widehat I$ yields
\[
\eta_{\omega,n}(f_\omega^2)\to \eta_{\omega,\infty}(f_\omega^2)
\]
as $n\to\infty$
for $m$-a.e.~$\omega\in\Omega$.
Hence, by Ces\`aro's theorem,
\[
\frac{1}{N} \sum_{n=0}^{N-1} \eta_{\sigma^n\omega,N-n}(f^2_{\sigma^n\omega})
-
\frac{1}{N} \sum_{n=0}^{N-1} \eta_{\sigma^n\omega,\infty}(f^2_{\sigma^n\omega})
\to 0
\]
for $m$-a.e.~$\omega\in\Omega$.
Then, by Birkhoff's ergodic theorem,
\[
\frac{1}{N} \sum_{n=0}^{N-1} \eta_{\sigma^n\omega,\infty}(f^2_{\sigma^n\omega})
\to
\int_\Omega \eta_{\omega,\infty}(f_\omega^2)\,dm(\omega)
\]
for $m$-a.e.~$\omega\in\Omega$.

For the off-diagonal term, write $k=j-i$.
Using again the conditional invariance of $\eta$, we have
\[
\eta_{\omega,N}\bigl(f_{\sigma^i\omega}\circ T_\omega^i \cdot f_{\sigma^{i+k}\omega}\circ T_\omega^{i+k}\bigr)
=
\eta_{\sigma^i\omega,N-i}\bigl(f_{\sigma^i\omega}\cdot f_{\sigma^{i+k}\omega}\circ T_{\sigma^i\omega}^k\bigr)
\]
for $m$-a.e.~$\omega\in\Omega$.
Therefore
\begin{align*}
&\frac{2}{N}\sum_{i=0}^{N-2}\sum_{j=i+1}^{N-1}
\eta_{\omega,N}\bigl(f_{\sigma^i\omega}\circ T_\omega^i \cdot f_{\sigma^j\omega}\circ T_\omega^j\bigr)\\
&=
2\sum_{k=1}^{N-1}\frac{N-k}{N}
\left(
\frac{1}{N-k}
\sum_{i=0}^{N-1-k}
\eta_{\sigma^i\omega,N-i}\bigl(f_{\sigma^i\omega}\cdot f_{\sigma^{i+k}\omega}\circ T_{\sigma^i\omega}^k\bigr)
\right).
\end{align*}
For each fixed $k\ge 1$, set
\[
a_k(\omega):=
\eta_{\omega,\infty}\bigl(f_\omega\cdot f_{\sigma^k\omega}\circ T_\omega^k\bigr).
\]
Then Lemma~\ref{expozeta} gives
\[
\eta_{\omega,n}\bigl(f_\omega\cdot f_{\sigma^k\omega}\circ T_\omega^k\bigr)
\to
a_k(\omega)
\qquad (n\to\infty)
\]
for $m$-a.e.~$\omega\in\Omega$.
Hence, again by Ces\`aro's theorem, for each fixed $k\ge 1$,
\[
\frac{1}{N-k}
\sum_{i=0}^{N-1-k}
\eta_{\sigma^i\omega,N-i}\bigl(f_{\sigma^i\omega}\cdot f_{\sigma^{i+k}\omega}\circ T_{\sigma^i\omega}^k\bigr)
-
\frac{1}{N-k}
\sum_{i=0}^{N-1-k}
a_k(\sigma^i\omega)
\to 0
\]
for $m$-a.e.~$\omega\in\Omega$.
Applying Birkhoff's ergodic theorem to $a_k$, we obtain
\[
\frac{1}{N-k}
\sum_{i=0}^{N-1-k}
a_k(\sigma^i\omega)
\to
\int_\Omega \eta_{\omega,\infty}\bigl(f_\omega\cdot f_{\sigma^k\omega}\circ T_\omega^k\bigr)\,dm(\omega)
\]
for each fixed $k\ge 1$ and for $m$-a.e.~$\omega\in\Omega$.
Therefore, for each fixed $k\ge 1$,
\begin{align*}
\frac{N-k}{N}
\left(
\frac{1}{N-k}
\sum_{i=0}^{N-1-k}
\eta_{\sigma^i\omega,N-i}\bigl(f_{\sigma^i\omega}\cdot f_{\sigma^{i+k}\omega}\circ T_{\sigma^i\omega}^k\bigr)
\right)\to
\int_\Omega \eta_{\omega,\infty}\bigl(f_\omega\cdot f_{\sigma^k\omega}\circ T_\omega^k\bigr)\,dm(\omega)
\end{align*}
for $m$-a.e.~$\omega\in\Omega$.

Since $\eta_{\omega,\infty}(f_\omega)=0$ for $m$-a.e.~$\omega\in\Omega$, Lemma~\ref{expozeta} with $u=f$, $v\equiv 1$, and $k=1$ gives
\[
\eta_{\omega,n}(f_\omega)\to \eta_{\omega,\infty}(f_\omega)=0
\qquad (n\to\infty)
\]
for $m$-a.e.~$\omega\in\Omega$.
Hence, by Theorem~\ref{FCB} applied to $g(x,y)=xy$, $p=1$, $n_1=0$, $n_2=k$, $\zeta=\eta$, there exist $C>0$ and $r\in(0,1)$ such that
\[
|a_k(\omega)|\le Cr^k
\qquad \text{for $m$-a.e.~$\omega\in\Omega$ and all $k\ge 1$.}
\]
In particular,
\[
\sum_{k=1}^\infty
\left|
\int_\Omega \eta_{\omega,\infty}\bigl(f_\omega\cdot f_{\sigma^k\omega}\circ T_\omega^k\bigr)\,dm(\omega)
\right|
<\infty.
\]
Therefore, by dominated convergence,
\begin{align*}
\lim_{N\to\infty}
\frac{2}{N}\sum_{i=0}^{N-2}\sum_{j=i+1}^{N-1}
\eta_{\omega,N}\bigl(f_{\sigma^i\omega}\circ T_\omega^i \cdot f_{\sigma^j\omega}\circ T_\omega^j\bigr)
=
2\sum_{k=1}^\infty
\int_\Omega \eta_{\omega,\infty}\bigl(f_\omega\cdot f_{\sigma^k\omega}\circ T_\omega^k\bigr)\,dm(\omega)
\end{align*}
for $m$-a.e.~$\omega\in\Omega$.
Combining the diagonal and off-diagonal parts proves item \textup{(1)}.

Item \textup{(2)} concerns the closed random dynamical system
$T_\omega:X_{\omega,\infty} \to X_{\sigma\omega,\infty}$
with its random invariant probability measure $\eta_{\omega,\infty}$.
This is exactly the setting used in the proof of \cite[Proposition~3]{DFGV18}, and the coboundary characterization follows.
\end{proof}

\begin{pro}\label{prop:0625b}
Let $f \in \mathrm{BV}_\Omega(I)$ with $\eta_{\omega,\infty}(f_\omega) = 0$ for $m$-a.e.~$\omega\in\Omega$. 
Then
\[
\lim_{N\to\infty} \frac{\sigma_{\omega,N}^2}{N} = \Sigma^2(f)
\]
for $m$-a.e.~$\omega\in\Omega$.
In particular, if $f$ is not a coboundary, then there exist a constant $C>0$ and a random variable $N_\omega\ge 1$ such that
\[
\sigma_{\omega,N} > C\sqrt N
\]
for $m$-a.e.~$\omega \in \Omega$ and every $N\ge N_\omega$.
\end{pro}

\begin{proof}
We denote the centered observables by
$\tilde{f}_n := f_{\sigma^n\omega} \circ T^n_\omega$ (centered with respect to $\eta_{\omega,\infty}$ since $\eta_{\omega,\infty}(f_\omega) = 0$) and
$\bar{f}_n := f_{\sigma^n\omega} \circ T^n_\omega - \eta_{\omega,N}(f_{\sigma^n\omega} \circ T^n_\omega)$
(centered with respect to $\eta_{\omega,N}$).
Let $S_N(\bar{f}) = \sum_{n=0}^{N-1} \bar{f}_n$ and $S_N(\tilde{f}) = \sum_{n=0}^{N-1} \tilde{f}_n$. 
Set
\[
\Delta_N(f) := S_N(\bar{f}) - S_N(\tilde{f}) = - \sum_{n=0}^{N-1} \eta_{\omega,N}(f_{\sigma^n\omega} \circ T^n_\omega).
\]
By Lemma \ref{expozeta}, 
$|\eta_{\omega,N}(f_{\sigma^n\omega} \circ T^n_\omega)| = |\eta_{\sigma^n\omega,N-n}(f_{\sigma^n\omega})| \le C_f \kappa^{N-n}$ for some $C_f > 0$, so
$|\Delta_N(f)| \le C_f \sum_{j=1}^{N} \kappa^j \le C_f \kappa/(1-\kappa)$,
which is uniformly bounded in $N$.

Expanding the square, we have
\begin{align*}
\frac{1}{N} \eta_{\omega,N}\left[(S_N(\bar{f}))^2\right]
&= \frac{1}{N} \eta_{\omega,N}\left[(S_N(\tilde{f}))^2\right] + \frac{2 \Delta_N(f)}{N} \eta_{\omega,N}\left[S_N(\tilde{f})\right] + \frac{(\Delta_N(f))^2}{N}.
\end{align*}
The third term vanishes as $N \to \infty$ since $\Delta_N(f)$ is bounded.
Moreover,
\[
\eta_{\omega,N}[S_N(\tilde{f})] = \sum_{n=0}^{N-1} \eta_{\omega,N}(f_{\sigma^n\omega} \circ T^n_\omega) = -\Delta_N(f), 
\]
so the second term equals $-2(\Delta_N(f))^2/N \to 0$.
The first term converges to $\Sigma^2(f)$ by Proposition~\ref{prop:0625a}. 
This completes the proof of the first assertion.

The second assertion follows immediately: if $f$ is not a coboundary, then $\Sigma^2(f) > 0$ by Proposition~\ref{prop:0625a}(2), and the convergence $\sigma_{\omega,N}^2/N \to \Sigma^2(f)$ yields the stated lower bound.
\end{proof}

\subsection{Completing the proof of Theorem \ref{CLT}}\label{s:6.4} For $N \ge 2$ and $\omega \in \Omega$, 
we set
$$
\tilde Y_n = 
\overline{f_{\omega,N,n}}, \quad \mu = \eta_{\omega,N}.
$$
By Theorem \ref{FCB}, there exists a full measure subset $\Omega_1 \subset \Omega$ such that, for 
every $\omega \in \Omega_1$, $( \tilde{Y}_n )_{0 \le n < N}$ satisfies 
 \eqref{eq:fcb} with constant $C_* = L\mathbf{C}_\omega$ and rate function $R(n) = r^n$.

Since $\overline{f_{\omega,N,n}}$ is centered with respect to $\eta_{\omega,N}$ (time-dependent centering), we apply Proposition~\ref{prop:0625b} to conclude that
$\sigma_{\omega,N}^2 / N \to \Sigma^2(f)$ as $N \to \infty$.
By Proposition~\ref{prop:0625a}(2), $\Sigma^2(f) > 0$ 
provided $f$ is not a coboundary.
In particular, $\sigma_{\omega,N}^{-1} = O(N^{-1/2})$.

Therefore, for $\omega \in \Omega_1$ we obtain the following estimates by applying Theorem \ref{thm:clt_tikhomirov} and Corollary \ref{cor:kolmogorov}: for $m$-a.e.~$\omega\in\Omega$ and every $N\ge N_\omega$,
\begin{align*}
d_{\mathcal W}( W_{\omega,N} ,Z) &\le C \mathbf{C}_\omega L^6  \sigma_{\omega, N}^{-3} N = O(N^{-1/2}), \\
d_\cK( W_{\omega,N} , Z) 
&\le C \mathbf{C}_\omega  L^4 \log^2(N + 1) (  N^{-1/2} + \sigma_{\omega, N}^{-1} + N \sigma_{\omega, N}^{-3} + \sqrt{N} \sigma_{\omega, N}^{-2}  ) = O(\log^2(N) N^{-1/2}),\\
d_{\mathcal{S}}(\mathscr W_{\omega,N},\mathscr Z) &\le  C L^5 (   \mathbf{C}_\omega^{3/2}  + 1  )   \sigma_{\omega, N}^{-3} N = O(  N^{-1/2}).
\end{align*}
Here, $C > 0$ is a constant independent of $\omega$ and $N$, and the $O(\cdot)$ bounds follow from $\sigma_{\omega,N}^{-2} = O(N^{-1})$ for $N\ge N_\omega$.
This completes the proof of Theorem~\ref{CLT}.

\section{Proof of Theorem \ref{thm:clt_tikhomirov}(2)}\label{s:5.4}

In this section, we prove Theorem \ref{thm:clt_tikhomirov}(2) by adapting techniques from \cite{T80, S84}.
Throughout this section, $C, C_1, C_2, \ldots$ denote absolute constants whose values may change from line to line.

Let $\{\tilde Y_n \}_{0\le n<N}$ be real-valued random variables on a probability space $(M, \cF, \mu)$ satisfying \eqref{eq:X_j}. Recall the definitions of $W$ and $\sigma_N^2$ from \eqref{eq:defs} and 
\eqref{eq:sigmas}, respectively,
and
write 
$$
\varphi(t) = \mu( e^{ i t W } )
$$
for the characteristic function of $W$.

\subsection{About Tikhomirov's method} Tikhomirov \cite{T80} introduced a method for estimating the Kolmogorov distance in the CLT for stationary weakly dependent sequences, inspired by ideas of Stein \cite{S72} but developed within the classical characteristic function framework of Berry and Esseen. This method was later extended to nonstationary sequences by Sunklodas \cite{S84}. We begin by briefly reviewing some of the key ideas in \cite{T80, S84}, which form the basis of the proof of Theorem~\ref{thm:clt_tikhomirov}(2).
For further background on Tikhomirov's method, as well as 
extensions of the method 
to target distributions other than the normal distribution, we refer the reader to \cite{AMPS16}.

Recall that $Y_n = \sigma_N^{-1} \tilde Y_n$.
For $0 \le j, K < N$, define
\begin{align}\label{eq:wj}
	\begin{split}
	W_{j}^{(\ell)} &=  \sum_{
		\substack{
			0 \le n < N \\
			|n - j | >  K \ell
	}}  Y_n, \quad \ell \ge 1,\\
	W_{j}^{(0)} &= W.
	\end{split}
\end{align}
Moreover, for $t \in \RR$, let
\begin{align*}
	\xi_j^{(\ell)} = \xi_j^{(\ell)}(t) =  e^{it  ( W_j^{(\ell - 1)}  - W_j^{(\ell)} ) } - 1.
\end{align*}
Then, the derivative of the characteristic function $\varphi(t) = \mu( e^{i t W} )$ of $W$ satisfies
\begin{align*}
	\varphi'(t) &=  \mu( iW e^{itW} )  = \sum_{j=0}^{N - 1}  \mu \left(  i Y_j   e^{it W_j^{(0)} }  \right).
\end{align*}
Adding and subtracting $\mu(i Y_j e^{itW_j^{(1)}})$ from the summands, we find that
\begin{align}\label{eq:decomp_phi_prime-1}
\varphi'(t) =  \sum_{j=0}^{N-1} \mu(i Y_j e^{itW_j^{(1)}})  
+ \sum_{j=0}^{N-1} \mu(  i Y_j \xi_{j}^{(1)}  e^{ it W_j^{(1)} } ).
\end{align}
Let $k \ge 2$. Successively
adding and subtracting 
$$\mu(i Y_j \xi_j^{(1)} e^{itW_j^{(2)}}), \ldots, \mu(iY_j \prod_{\ell = 1 }^{k-1} \xi_j^{(\ell)} e^{itW_j^{(k)}} )$$ 
from the summands of the second term in \eqref{eq:decomp_phi_prime-1}, one obtains the representation
\begin{align}\label{eq:decom_phi_prime}
	\begin{split}
		\varphi'(t) &=  
		\sum_{j=0}^{N-1} \mu(i Y_j e^{itW_j^{(1)}}) + \sum_{j=0}^{N-1} \sum_{r=2}^k \mu \biggl(iY_j \prod_{\ell = 1 }^{r-1} \xi_j^{(\ell)} e^{itW_j^{(r)}} \biggr) \\
		&+ \sum_{j=0}^{N-1}  
		\mu \biggl(iY_j \prod_{\ell = 1 }^{k} 
		\xi_j^{(\ell)} e^{itW_j^{(k)}} \biggr).
	\end{split}
\end{align}

Further, for $r \ge 0$, define
\begin{align}\label{eq:eta}
	\eta_j^{(r)} =  e^{-i t \hat{W}_j^{(r)} } - 1, \quad \hat{W}_j^{(r)} =  W  - W_j^{(r)} = \sum_{ \substack{
			0 \le n < N \\
			| n - j | \le Kr 
	}  } Y_n.
\end{align}

The following two lemmas are taken from \cite{S84}. 
We include their proofs to make the exposition self-contained.

\begin{lem}\label{lem:sunklodas_1} 
	Denote
	\begin{align*}
		a_j^{(r)} =  \mu \biggl( i Y_j \prod_{\ell = 1}^r \xi_j^{(\ell)}  \biggr).
	\end{align*}
	Then,
	\begin{align*}
		\varphi'(t) = (E_1 + E_2) \varphi(t) + \sum_{i=3}^6 E_i,
	\end{align*}
	where
	\begin{align*}
		E_1 &= \sum_{j=0}^{N-1} a_j^{(1)}, \quad 
		E_2 = \sum_{j=0}^{N-1} \biggl[ a_j^{(1)} \mu( \eta_j^{(2)} ) + \sum_{r=3}^k a_j^{(r-1)}  \mu( \eta_j^{(r)} + 1 )  \biggr], \\
		E_3 &= \sum_{r=2}^k \sum_{j=0}^{N-1}  a_j^{(r-1)} \mu[  (  \eta_j^{(r)} - \mu (  \eta_j^{(r)} ) ) e^{itW } ], \quad 
		E_4 = \sum_{j=0}^{N-1} \mu \biggl(   i Y_j \prod_{\ell = 1}^k \xi_j^{(\ell)} e^{it W_j^{(k)} }  \biggr), \\
		E_5 &= \sum_{r = 2}^k \sum_{j=0}^{N-1} \mu \biggl(  i Y_j \prod_{\ell = 1}^{r-1} \xi_j^{(\ell)} ( e^{it W_j^{(r)}} - \mu ( e^{it W_j^{(r)}}  ) )  \biggr), 
		\quad E_6 = \sum_{j=0}^{N-1} \mu( i Y_j e^{it W_j^{(1)} } ).
	\end{align*}
	
\end{lem}

\begin{proof} Starting from \eqref{eq:decom_phi_prime}, we have 
	\begin{align*}
		\varphi'(t) &=   E_4 + E_6 + \sum_{j=0}^{N-1} \sum_{r=2}^k \mu \biggl(iY_j \prod_{\ell = 1 }^{r-1} \xi_j^{(\ell)} e^{itW_j^{(r)}} \biggr) \\
		&= E_4 + E_5 +  E_6 + \sum_{j=0}^{N-1} \sum_{r=2}^k \mu \biggl(iY_j \prod_{\ell = 1 }^{r-1} \xi_j^{(\ell)}  \mu ( e^{itW_j^{(r)}} ) \biggr).
	\end{align*}
	Since
	\begin{align*}
		\mu( e^{it W_{j}^{(r)} } ) &= \varphi(t)  + \mu( e^{it W_{j}^{(r)} } ) - \varphi(t)  \\
		&= \mu(  \eta_j^{(r)} + 1  ) \cdot  \varphi(t)  + \mu(   (   \eta_j^{(r)} - \mu (  \eta_j^{(r)} ) ) e^{itW} ),
	\end{align*}
	it follows that 
	\begin{align*}
		\varphi'(t) &= \sum_{i=4}^6 E_i 
		+ \sum_{j=0}^{N-1} \sum_{r=2}^k \mu \biggl(iY_j \prod_{\ell = 1 }^{r-1} \xi_j^{(\ell)}  \mu(  \eta_j^{(r)} + 1  ) \cdot  \varphi(t)   \biggr) \\
		&+ \sum_{j=0}^{N-1} \sum_{r=2}^k \mu \biggl(iY_j \prod_{\ell = 1 }^{r-1} \xi_j^{(\ell)}   \mu[   (   \eta_j^{(r)} - \mu (  \eta_j^{(r)} ) ) e^{itW} ] \biggr) \\
		&=  \sum_{i=3}^6 E_i + \sum_{j=0}^{N-1} \sum_{r=2}^k a_j^{(r - 1)} \mu(  \eta_j^{(r)} + 1  )  \varphi(t)  
		=   (E_1 + E_2) \varphi(t) + \sum_{i=3}^6 E_i,
	\end{align*}
	as wanted.
\end{proof}

\begin{lem}\label{lem:sunklodas_2} Suppose that
	\begin{align}\label{eq:ode}
		f'(t) = (-t + \theta(t) a (t) ) f(t) + \theta(t) b(t), \quad f(0) = 1,
	\end{align}
	for $|t| \le T_1$, where
	\begin{align*}
		a(t) &= a^{(0)} + a^{(1)}|t| + a^{(2)}t^2, \\
		b(t) &= b^{(0)} + b^{(2)}t^2.
	\end{align*}
	Here $a^{(i)} \ge 0$ and $b^{(i)} \ge 0$ for $0 \le i \le 2$ are constants independent of $t$ and $\theta(t)$ 
	is a complex valued function with $| \theta(t) | \le 1$. If $a^{(1)} \le 1/6$, then
	for an absolute constant $C > 0$,
	\begin{align*}
		&| f(t) - e^{- t^2/2} | 
		\le C \biggl[ a^{(0)}|t| + a^{(1)}t^2 + a^{(2)}|t|^3    \biggr] e^{-t^2 / 4}
		+ C e \biggl[  b^{(0)} \min \{    |t|^{-1}, |t| \} +   b^{(2)} |t|  \biggr]
	\end{align*}
	holds for all $|t| \le \min \{ T_1, T_2 \}$ where
	\begin{align*}
		T_2 = \min \biggl\{  \frac{1}{a^{(0)}}, \frac{1}{6a^{(2)}}      \biggr\}.
	\end{align*}
\end{lem}

\begin{proof} The linear ODE \eqref{eq:ode} is solved by
	\begin{align*}
		f(t) = \exp\biggl\{  -\frac12 t^2 + x_0(t) \biggr\}
		+ \exp\biggl\{  -\frac12 t^2 + x_0(t) \biggr\} 
		\int_0^t \theta(u) b(u)  \exp\biggl\{  \frac12 u^2 - x_0(u) \biggr\}  \, du,
	\end{align*}
	where
	\begin{align*}
		x_0(t) = \int_0^t \theta(u) a(u) \, du.
	\end{align*}
	Since $|\theta(t)| \le 1$, we have $| f(t) - e^{-\frac12 t^2} | \le I + II$, where 
	\begin{align*}
		I &= |x_0(t)| e^{ -\frac12 t^2 + |x_0(t)|  }, \\
		II &=  e^{-\frac12 t^2} \int_0^{|t|} b(u) \exp \biggl\{ \frac{u^2}{2} 
		+ \int_{u}^{|t|} a(v) \, dv \biggr\} \, du.
	\end{align*}
	
	Using $a^{(1)} \le 1/6$, we have for any $|t| \le T_2$ the inequality
	\begin{align*}
		|x_0(t)|
		\le a^{(0)} |t| + \frac12 a^{(1)} t^2 + \frac13 a^{(2)} |t|^3
		\le 1 + \frac{5}{36} t^2,
	\end{align*}
	which implies
	\begin{align*}
		I \le  e ( a^{(0)} |t| + a^{(1)} t^2 + a^{(2)} |t|^3 ) e^{ - t^2 / 4 }.
	\end{align*}
	
	For $II$, observe that whenever $0 \le u \le t \le T_2$,
	\begin{align}\label{eq:II_1}
		\begin{split}
			\int_u^t a(v) \, dv &= \int_{u}^t a^{(0)} + a^{(1)} v + a^{(2)} v^2 \, dv 
			\le 1 + \int_u^t   v ( a^{(1)}  + a^{(2)} v ) \, dv \\
			&\le 1 + t  \biggl(  a^{(1)} (t - u)  +  \frac{ a^{(2)} }{2} (t^2 - u^2)   \biggr) \\
			&\le 1 + \frac16 ( t^2 - u^2 ) + \frac{1}{12} (t^2 - u^2) = 1 + \frac{1}{4} (t^2 - u^2),
		\end{split}
	\end{align}
	where $t \le 1 / a^{(0)}$ was used in the second step and $a^{(1)} \le 1/6$, 
	$t \le 1/6a^{(2)}$ were used in the second-to-last step. Note that
	\begin{align}\label{eq:II_2}
		\int_0^{|t|} u^2 e^{u^2 / 4} \, du \le 2 |t| e^{ t^2 / 4 }
		\qquad \text{and} \qquad 
		\int_0^{|t|} e^{ u^2 / 4 } \, du \le \min \{  4 / |t| , |t| \} e^{ t^2 / 4 }.
	\end{align}
	Consequently, for $|t| \le T_2$,
	\begin{align*}
		II 
		&\le e^{-\frac12 t^2} \int_0^{|t|} b(u) \exp \biggl\{ \frac{u^2}{2} \biggr\}  \exp \biggl\{ 
		1 + \frac{1}{4} (t^2 - u^2)
		\biggr\} \, du \\
		&= e e^{-\frac14 t^2} \int_0^{|t|} ( b^{(0)} + b^{(2)}u^2 ) e^{  \frac14 u^2  }  \, du \\
		&\le e e^{-  \frac14 t^2 } \biggl( 
		b^{(0)}  \min \{  4 / |t| , |t| \} e^{ \frac14 t^2 } + b^{(2)} 2 |t| e^{ \frac14 t^2 }
		\biggr) \\
		&\le  C e  \biggl( 
		b^{(0)}  \min \{  4 / |t| , |t| \}  + b^{(2)}  |t|
		\biggr),
	\end{align*}
	where \eqref{eq:II_1} was used in the first step, and \eqref{eq:II_2} was used in the second-to-lasts step.
	This completes the proof of the lemma.
\end{proof}

In the proof of Theorem \ref{thm:clt_tikhomirov}(2),
Lemma \ref{lem:sunklodas_2} will be applied 
as follows. Suppose that 
\begin{align}\label{eq:phi_ode}
	\varphi'(t) = (-t + \theta(t) a (t) ) \varphi (t) + \theta(t) b(t)
\end{align}
for $|t| \le T_1$, where $\theta(t), a(t), b(t)$ are as in Lemma \ref{lem:sunklodas_2}. Observe that 
$\varphi'(0) =  \mu(i W) = 0$. By Esseen's inequality,
\begin{align*}
	d_\cK( W, Z) &\le C \biggl( \frac{1}{ \min \{ T_1, T_2 \}  } + \int_{ 0  }^{ \min \{ T_1, T_2 \}  } | \varphi(t) - e^{-t^2/2} | \frac{1}{t} dt \biggr).
\end{align*}
Therefore, whenever $a^{(1)} \le 1/6$, Lemma \ref{lem:sunklodas_2} yields
\begin{align}\label{eq:bound_on_kolmogov}
	d_\cK( W, Z)  &\le C( T_1^{-1} + a^{(0)} + a^{(1)} + a^{(2)} + b^{(0)} + b^{(2)}T_1).
\end{align}
Hence, the goal is to establish \eqref{eq:phi_ode} 
and to estimate the constants $a^{(i)}, b^{(i)}$. To this end, we 
use Lemma \ref{lem:sunklodas_1} and (FCB).

\subsection{Proof of Theorem \ref{thm:clt_tikhomirov}(2)}
Let 
$$
2 + \frac{1}{2} \log (N) /\log(2)  < K \le N
$$
and recall the definitions of $W_j^{(\ell)}$, $\eta_j^{(r)}$, and
 $\hat{W}_j^{(\ell)}$ in \eqref{eq:wj} and \eqref{eq:eta}.
We will only need (FCB) for a particular subclass of 
separately Lipschitz continuous
complex-valued
functions $g$, which is described by (A1)-(A4) in the 
following result. Recall from Remark \ref{rem:fcb} that \eqref{eq:fcb} extends 
to complex valued functions $g$ as well as to the estimate 
\eqref{eq:fcb_many_gaps}.

\begin{pro}\label{lem:conditions} Assume \eqref{eq:fcb}
and define $\hat R$ as in Theorem \ref{thm:clt_tikhomirov}. Then the following hold for any $t \in \RR$.
	\begin{itemize}
		\item[(A1)]{For all $0  \le i < j < N$,
			\begin{align*}
				| \mu( \tilde Y_i \tilde Y_j ) | \le C C_* L^2  R(j-i).
		\end{align*}
	}
		\item[(A2)] Whenever, $0 \le j,\ell < N$ satisfy $|\ell - j| > 2 K r$ where $r \ge 1$,
		\begin{align*}
			|  \mu( \eta_j^{(r)}  (\eta_\ell^{(r)} )^* )  -  \mu( \eta_j^{(r)})  \mu ( (\eta_\ell^{(r)} )^* ) | \le C C_* t^2 \sigma_N^{-2} L^2 (Kr)^2 R ( | \ell - j | - 2Kr).
		\end{align*}
		Here, $z^*$ denotes the complex conjugate of 
        $z \in \CC$. 
		\item[(A3)]{For all $r \ge 2$ and $0 \le j < N$,
			\begin{align*}
				&\biggl| \mu \biggl[  i Y_j \prod_{\ell = 1}^{r-1} \xi_j^{(\ell)} ( e^{it W_j^{(r)}} - \mu ( e^{it W_j^{(r)}} ) )  \biggr] \biggr| \\ 
				&\le C C_* L ( 1 + |t| \sigma_N^{-1} ) \sigma_N^{-1} (2K L  |t| \sigma_N^{-1})^{r-1}  \biggl\{   \hat{R} (  K  ) 
				+   R(2rK + 2) \bigg\}.
			\end{align*}
		}
		
		\item[(A4)]{For all $0 \le j < N$,
			\begin{align*}
				| \mu( Y_j e^{it W_j^{(1)} } ) |  \le C_* L (  \sigma_N^{-1}  +  \sigma_N^{-2}   |t| ) \hat R (K ).
			\end{align*}
		}
		
	\end{itemize}
	
\end{pro}

\begin{proof} Throughout the proof we use the conventions introduced at the beginning 
of the proof of Theorem \ref{thm:clt_tikhomirov} in Section \ref{s:5.2} for elements $x \in [-L, L]^{J}$ where $J \subset \ZZ$.

\smallskip 
	
\noindent \textbf{(A1):} The desired inequality follows from \cite[Proposition 5.6]{LNN25}.
	
\smallskip 
	
	\noindent \textbf{(A2):} 
	For $J = \{  0 \le n < N \: : \: |n - j| \le Kr  \}
    \cup 
    \{  0 \le n < N \: : \: |n - \ell| \le Kr  \}
    $, define 
	\begin{align*}
		g_1(x_J)
		= \biggl\{   \exp\biggl(  -it  \sigma_N^{-1} \sum_{ \substack{0 \le n < N \\  |n-j| \le Kr } } x_n   \biggr)  - 1  \biggr\} \biggl\{   \exp\biggl(  it  \sigma_N^{-1} \sum_{ \substack{
				0 \le n < N \\	 |n-\ell| \le Kr } } x_n   \biggr)  - 1  \biggr\}.
	\end{align*}
	Using $|e^{ix} - 1| \le |x|$, 
	a straightforward computation shows that
	\begin{align*}
		\Vert g_1 \Vert_{  \text{Lip} } \le C t^2 \sigma_N^{-2} L^2 (Kr)^2.
	\end{align*}
	Since
	$
	g_1(\tilde Y_{J} ) = \eta_j^{(r)} (\eta_\ell^{(r)} )^* ,
	$
	we see that (A2) follows by an application of (FCB) with $g_1$ 
	as above.
	\smallskip 
	
	\noindent\textbf{(A3):} 
	We write $W_j^{(r)} = U_j^{(r)} + V_j^{(r)}$ where
	\begin{align*}
		U_j^{(r)} = \sum_{n = 0}^{j - rK -1}  Y_n \quad \text{and} \quad  V_j^{(r)} = \sum_{n = j + rK + 1}^{N - 1}  Y_n.
	\end{align*}
	Next, we decompose 
	\begin{align*}
		\mu \biggl[  i Y_j \prod_{\ell = 1}^{r-1} \xi_j^{(\ell)} ( e^{it W_j^{(r)}} - \mu ( e^{it W_j^{(r)}}  ) )  \biggr] = I + II + III,
	\end{align*}
	where:
	\begin{align*}
		I &= \mu \biggl[  i Y_j \prod_{\ell = 1}^{r-1} \xi_j^{(\ell)}  e^{it U_j^{(r)}}e^{it V_j^{(r)}}  \biggr] - \mu \biggl[  i Y_j \prod_{\ell = 1}^{r-1} \xi_j^{(\ell)}  e^{it U_j^{(r)}} \biggr]  \mu[ e^{it V_j^{(r)}} ], \\
		II &= \mu[ e^{it V_j^{(r)}} ]  \biggl( \mu \biggl[  i Y_j \prod_{\ell = 1}^{r-1} \xi_j^{(\ell)}  e^{it U_j^{(r)}} \biggr]  - \mu \biggl[  i Y_j \prod_{\ell = 1}^{r-1} \xi_j^{(\ell)}  \biggr]  \mu[e^{it U_j^{(r)}} ]     \biggr), \\
		III &= \mu \biggl[  i Y_j \prod_{\ell = 1}^{r-1} \xi_j^{(\ell)}  \biggr] \biggl(   \mu[e^{it U_j^{(r)}} ] \mu[e^{it V_j^{(r)}} ]
		- \mu[e^{it U_j^{(r)}} e^{it V_j^{(r)}} ] 
		\biggr).
	\end{align*}

	For
	$$
	J = \ZZ_+ \cap \biggl( [0, j - r K ) \cup [j - K(r-1), j + K(r-1)] \cup (j+ rK, N)  \biggr),
	$$
	define
	\begin{align*}
		&g_2( x_J ) =   i  \sigma_N^{-1} x_j \prod_{\ell = 1}^{r-1} \biggl(   \exp \biggl\{ it \sum_{ 
			\substack{0 \le n < N \\  K(\ell - 1) < |n-j| \le K \ell } } \sigma_N^{-1} x_n \biggr\} - 1   \biggr) e^{ it    \sum_{n = 0}^{j - rK -1}  \sigma_N^{-1} x_n }  e^{it \sum_{n = j + rK + 1}^{N - 1}  \sigma_N^{-1} x_n}.
	\end{align*}
	Then, 
	\begin{align*}
		\Vert g_2 \Vert_{ \text{Lip}} \le C L \sigma_N^{-1}  (   2K L |t| \sigma_N^{-1}  )^{r-1}.
	\end{align*}
	It follows by (FCB) that 
	\begin{align*}
		|I| \le  C C_* L \sigma_N^{-1}  (2K L  |t| \sigma_N^{-1})^{r-1} R(  K + 1 ).
	\end{align*}
	Similarly,
	\begin{align*}
		|II| &\le C C_* \Vert g_2 \Vert_{ \text{Lip} } R(K + 2) \le 
		C C_* L \sigma_N^{-1}  (2K L  |t| \sigma_N^{-1})^{r-1} \hat{R}(  K  ).
	\end{align*}
	
	Note that $III = 0$ if $j + rK + 1 > N - 1$. Otherwise, define
	$$
	g_3(x_J) = \exp\biggl(  -it  \sigma_N^{-1}  \sum_{n = 0}^{j - rK -1}  x_n   \biggr)   \exp\biggl(  it  \sigma_N^{-1} 
	\sum_{n = j + rK + 1}^{N - 1} 
	x_n   \biggr)  ,
	$$
	which 
	satisfies $\Vert g_3 \Vert_{\text{Lip}} \le 1 + |t| \sigma_N^{-1}$. Using this inequality together with 
	$\Vert \xi_j^{(\ell)} \Vert_\infty \le |t| \sigma_N^{-1} L 2 K$ and  \eqref{eq:fcb}, we obtain
	\begin{align*}
		|III| &\le   L  \sigma_N^{-1}   \biggl[  \prod_{\ell = 1}^{r-1} \Vert \xi_j^{(\ell)} \Vert_\infty  \biggr] \biggl|  \mu[e^{it U_j^{(r)}} ] \mu[e^{it V_j^{(r)}} ]
		- \mu[e^{it U_j^{(r)}} e^{it V_j^{(r)}} ]  \biggr| \\
		&\le  L  \sigma_N^{-1} ( |t| \sigma^{-1}_N L 2 K  )^{r-1}  C_* ( 1 + |t| \sigma_N^{-1} ) R(2rK + 2).
	\end{align*}
	
	\noindent\textbf{(A4):} For 
	$$
	g_4(x_J) = \sigma_N^{-1}  x_j \exp \biggl\{ it \sigma_N^{-1}  \sum_{
		\substack{
			0 \le n < N \\
			|n - j | >  K
	}}  x_n \biggr\}
	$$
	with 
	$
	J = \{  0 \le n < N \: : \:  |n - j | > K \} \cup \{ j \}
	$,
	we have 
	$$
	\Vert g_4 \Vert_{\text{Lip}} \le  \sigma_N^{-1} L +  \sigma_N^{-2} L  |t|.
	$$
	Since $g_4(\tilde Y_J) = Y_j e^{it W_j^{(1)} }$, 
	it follows from \eqref{eq:fcb_many_gaps} that 
	\begin{align*}
		| \mu( Y_j e^{it W_j^{(1)} } ) |  \le C C_* (  \sigma_N^{-1} L +  \sigma_N^{-2} L  |t| ) \hat R (K).
	\end{align*}
	
\end{proof}

Theorem \ref{thm:clt_tikhomirov}(2) follows immediately from Proposition \ref{lem:conditions}
and the following theorem.

\begin{thm}\label{thm:tikhomirov_aux}
	Conditions (A1)-(A4) in Proposition \ref{lem:conditions}
	imply \eqref{eq:clt_tikhomirov}.
\end{thm}

\begin{proof}
Recall the following decomposition from Lemma \ref{lem:sunklodas_1}:
\begin{align}\label{eq:decomp_proof}
	\varphi'(t) = (E_1 + E_2) \varphi(t) + \sum_{i=3}^6 E_i.
\end{align}
We will estimate each term $E_i$ separately using (A1)-(A4). Throughout this proof, $\theta(t)$ 
denotes a generic complex valued function with $|\theta(t)| \le 1$. The definition of $\theta(t)$ may 
change from one line to the next. 
\smallskip 

\noindent($E_1$): For all $x \in \RR$ and $k \ge 0$,
\begin{align}\label{eq:exp_approx}
	\biggl| e^{ix} - \sum_{\nu = 0}^{k} \frac{(ix)^\nu }{\nu !} \biggr| \le \frac{ |x|^{k+1} }{(k + 1)!}.
\end{align}
Applying \eqref{eq:exp_approx} with $k=1$, we obtain
\begin{align*}
	E_1 = \sum_{j=0}^{N-1} \mu \biggl(  i Y_j  \xi_j^{(1)}  \biggr) = -t  \sum_{j=0}^{N-1} \mu(Y_j \hat{W}_j^{(1)} ) + \theta(t) C  t^2  N L^3 K^2 \sigma_N^{-3}.
\end{align*}
Since $\mu(W^2) = 1$ 
and $W = W_j^{(1)} + \hat{W}_j^{(1)}$, we have 
\begin{align*}
	\sum_{j=0}^{N-1} \mu(Y_j \hat{W}_j^{(1)} )  = 1 - \sum_{j=0}^{N-1} \mu(Y_j W_j^{(1)} ).
\end{align*}
By (A1),
\begin{align*}
	\biggl| \sum_{j=0}^{N-1} \mu(Y_j W_j^{(1)} )  \biggr| \le \sigma_N^{-2} \sum_{j=0}^{N-1} \sum_{  \substack{0 \le n < N \\ | n -j | > K }   } | \mu(\tilde Y_j \tilde Y_n) | \le C C_* L^2 \sigma_N^{-2} N \sum_{j=K}^{N-1} R(j).
\end{align*}
It follows that, with a redefined $\theta(t)$,
\begin{align*}
	E_1 = -t +  \theta(t) C_1 \biggl\{  |t| L^2  \sigma_N^{-2} N \sum_{j=K}^{N-1} R(j) + C_* t^2  N L^3 K^2 \sigma_N^{-3} \biggr\}.
\end{align*}

\noindent$(E_4)$:
By $\eqref{eq:exp_approx}$,
\begin{align*}
	\Vert \xi_j^{(\ell)} \Vert_\infty = \Vert e^{it  ( W_j^{(\ell - 1)}  - W_j^{(\ell)} ) } - 1 \Vert_\infty  \le  |t| \sigma_N^{-1} L 2 K
\end{align*}
and
\begin{align}\label{eq:eta_bound}
	\Vert \eta_j^{(r)} \Vert_\infty = \Vert e^{-i t \hat{W}_j^{(r)} } - 1 \Vert_\infty  \le  |t| \sigma_N^{-1} L 3Kr.
\end{align}
Hence,
\begin{align}\label{eq:a_j_bound}
	| a_j^{(r)} | \le \mu \biggl| i Y_j \prod_{\ell = 1}^{r} \xi_j^{(\ell)}  \biggr| \le \sigma_N^{-1} L ( |t| \sigma_N^{-1} L 2K  )^r = : u^{(r)}.
\end{align}
Consequently, 
\begin{align*}
	|E_4| &= \biggl|  \sum_{j=0}^{N-1} \mu \biggl(   i Y_j \prod_{\ell = 1}^k \xi_j^{(\ell)} e^{it W_j^{(k)} }  \biggr)  \biggr|
	\le N u^{(k)}
	\le N \sigma_N^{-1} L ( |t| \sigma_N^{-1} L 2K  )^k
	\le C_4 t^2 \sqrt{N} \sigma_N^{-3} L^3 K^2,
\end{align*}
whenever
\begin{align}\label{eq:ass_k_t}
	k \ge 2 + \frac{1}{2} \frac{ \log (N) }{ \log(2) } \quad \text{and} \quad  |t| \le  4^{-1} \sigma_N L^{-1} K^{-1} =: T_1
\end{align}
From now on we fix
$$
2 + \frac{1}{2} \frac{ \log (N) }{ \log(2) }  \le k \le K,
$$
which is possible by our assumption on $K$.
\smallskip 

\noindent\noindent$(E_2)$:
For $|t| \le T_1$, using \eqref{eq:eta_bound} and \eqref{eq:a_j_bound} we obtain
\begin{align*}
	|E_2| &= \biggl| \sum_{j=0}^{N-1} \biggl[ a_j^{(1)} \mu( \eta_j^{(2)} ) + \sum_{r=3}^k a_j^{(r-1)}  \mu( \eta_j^{(r)} + 1 )  \biggr] \biggr| 
	\le \sum_{j=0}^{N-1}   | a_j^{(1)} |  \Vert \eta_j^{(2)} \Vert_\infty + N \sum_{r=3}^k  | u^{(r-1)} |   \\
	&\le  C_2 N L^3 t^2 \sigma_N^{-3} K^2 + N \sigma_N^{-1} L ( |t| \sigma_N^{-1} L 2K  )^2 \sum_{r=0}^\infty   ( |t| \sigma_N^{-1} L 2K  )^r \le C_2 N \sigma_N^{-3}  L^3 K^2 t^2 .
\end{align*}

\noindent$(E_3)$: Denoting $\overline{\eta}_j^{(r)} = \eta_j^{(r)} - \mu(\eta_j^{(r)})$, by Jensen's inequality we have 
\begin{align*}
	\biggl|  \sum_{j=0}^{N-1} a_j^{(r-1)} \mu[  \overline{\eta}_j^{(r)} e^{itW } ] \biggr| 
	\le \biggl\{  \mu \biggl[ \biggl|   \sum_{j=0}^{N - 1}  a_j^{(r-1)}    \overline{\eta}_j^{(r)} \biggr|^2  \biggr]  \biggr\}^{1/2} \le 
	\biggl\{  
	\sum_{j=0}^{N-1} \sum_{p=0}^{N-1}  | a_j^{ (r-1) } a_p^{ (r-1) } | | \mu[    \overline{\eta}_j^{(r)}  ( \overline{\eta}_p^{(r)} )^* ] | 
	\biggr\}^{1/2}.
\end{align*}
Using \eqref{eq:a_j_bound}, \eqref{eq:eta_bound}, and (A2), we thus obtain the following for any $|t| \le T_1$:
\begin{align*}
	&|E_3| = \biggl| \sum_{r=2}^k \sum_{j=0}^{N-1}  a_j^{(r-1)} \mu[    \overline{\eta}_j^{(r)}  e^{itW } ] \biggr| \\
	&\le \sum_{r=2}^k  u^{(r-1)} \biggl[  \sum_{j=0}^{N-1} \sum_{p=0}^{N-1}  |  \mu( \eta_j^{(r)} ( \eta_p^{(r)}  )^*   )  -  \mu( \eta_j^{(r)})  \mu ( (\eta_p^{(r)} )^* ) | \biggr]^{1/2} \\
	&\le C_3 \sum_{r=2}^k  u^{(r-1)} \biggl[  \sum_{j=0}^{N-1} \sum_{  \substack{
			0 \le p < N \\
			|p-j| \le 2rK
	}  }  (t \sigma_N^{-1} L 2 Kr )^2
	+  \sum_{j=0}^{N-1} \sum_{ \substack{
			0 \le p < N \\
			|p-j| > 2rK
	}  }  C_*  t^2 \sigma_N^{-2} L^2 (Kr)^2 R( | p - j | - 2Kr)
	\biggr]^{1/2} \\
	&\le  C_3 K^{5/2} \sqrt{N} \sigma_N^{-3} L^3 t^2 \sum_{r=2}^k r^{\frac32}  2^{ - r + 2}   \\
	&+ C_3 \sqrt{C_*} t^2 \sigma_N^{-3} L^3 K^{5/2} \sum_{r=2}^k r  2^{ - r + 1}  \biggl( \sum_{j=0}^{N-1} \sum_{ \substack{
			0 \le p < N \\
			|p-j| > 2rK
	}  } R ( | p - j | - 2Kr) \biggr)^{1/2}.
\end{align*}
By H\"{o}lder's inequality,
\begin{align*}
	&\sum_{r=2}^k   r  2^{ - r }  \biggl( \sum_{j=0}^{N-1} \sum_{ \substack{
			0 \le p < N \\
			|p-j| > 2rK
	}  } R ( | p - j | - 2Kr) \biggr)^{1/2}\\
	&\le \biggl(  \sum_{r=2}^\infty r^2  4^{-r}  \biggr)^{1/2} \biggl(  \sum_{r=2}^k \sum_{j=0}^{N-1} \sum_{ \substack{
			0 \le p < N \\
			|p-j| > 2rK
	}  } R ( | p - j | - 2Kr)    \biggr)^{1/2} 
	\le C_3 \biggl( KN \sum_{j=1}^{N-1} R (j) \biggr)^{1/2},
\end{align*}
so that
$$
|E_3| \le C_3 \sqrt{C_*} \sqrt{N} \sigma_N^{-3} L^3 K^{5/2} t^2 \biggl\{ 1 + \biggl(  \sum_{j=1}^{N-1} R (j) \biggr)^{1/2} \biggr\}.
$$

\noindent$(E_5)$: For $|t| \le T_1$, we use (A3) to obtain
\begin{align*}
	| E_5 | &= \biggl| \sum_{r = 2}^k \sum_{j=0}^{N-1} \mu \biggl[  i Y_j \prod_{\ell = 1}^{r-1} \xi_j^{(\ell)} ( e^{it W_j^{(r)}} - \mu (  e^{it W_j^{(r)}} )  ) \biggr] \biggr| \\
	&\le   \sum_{r = 2}^k N C_5 C_* L ( 1 + |t| \sigma_N^{-1} ) \sigma_N^{-1} (2K L  |t| \sigma_N^{-1})^{r-1}  \biggl\{   \hat R(  K  ) 
	+   \hat R(2rK + 2) \bigg\}  \\
	&\le C_5 C_* L( 1 + |t| \sigma_N^{-1}   ) N \sigma_N^{-1} \hat{R}(K) 
	\le  C_5 C_* L N \sigma_N^{-1} \hat{R}(K).
\end{align*}

\noindent$(E_6)$: For $|t| \le T_1$, we use (A4) to obtain
\begin{align*}
	| E_6 |  &= \biggl| \sum_{j=0}^{N-1} \mu ( Y_j e^{it W_j^{(1)} } ) \biggr| \le C_6 C_* L N (  \sigma_N^{-1}  +  \sigma_N^{-2}   |t| )  \hat{R} (K) \le  C_6 C_* L N \sigma_N^{-1}    \hat{R} (K).
\end{align*}

Combining \eqref{eq:decomp_proof} with the foregoing estimates on $E_i$, it follows that 
\begin{align*}
	\varphi'(t) = (-t + \theta(t) a (t) ) \varphi (t) + \theta(t) b(t)
\end{align*}
holds whenever $|t| \le T_1$, where
\begin{align*}
	a(t) &= C_1 \biggl\{  |t| L^2  \sigma_N^{-2} N \sum_{j=K}^{N-1} R(j) + C_* t^2  N L^3 K^2 \sigma_N^{-3} \biggr\}
	+ C_2 N \sigma_N^{-3}  L^3 K^2 t^2 
	= a^{(1)} |t| + a^{(2)}t^2, \\
	a^{(1)} &= C_1 L^2  \sigma_N^{-2} N \sum_{j=K}^{N-1} R(j), \quad a^{(2)} = (C_1 + C_2) C_* L^3 K^2 N \sigma_N^{-3},
\end{align*}
and
\begin{align*}
	b(t) &= C_4 t^2 \sqrt{N} \sigma_N^{-3} L^3 K^2 
	+ C_3 \sqrt{C_*} \sqrt{N} \sigma_N^{-3} L^3 K^{5/2} t^2 \biggl\{ 1 + \biggl(  \sum_{j=1}^{N-1} R (j) \biggr)^{1/2} \biggr\} \\
	&+ C_5 C_* L N \sigma_N^{-1} \hat{R}(K) +  C_6 C_* L N \sigma_N^{-1}    \hat{R} (K) \\
	&= b^{(0)} + b^{(2)}t^2, \\
	b^{(0)} &=  ( C_5 + C_6) C_* L N \sigma_N^{-1} \hat{R}(K), \\
	b^{(2)} &= C_4  \sqrt{N} \sigma_N^{-3} L^3 K^2 
	+  C_3 \sqrt{C_*} \sqrt{N} \sigma_N^{-3} L^3 K^{5/2} \biggl\{ 1 + \biggl(  \sum_{j=1}^{N-1} R (j) \biggr)^{1/2} \biggr\}.
\end{align*}
Note that \eqref{eq:clt_tikhomirov} holds automatically if 
$a^{(1)} > 1/6$. Thus, we can assume that $a^{(1)} \le 1/6$.
Substituting $a^{(1)}$, $a^{(2)}$, $b^{(0)}$, $b^{(2)}$ into
\eqref{eq:bound_on_kolmogov}, we have 
\begin{align*}
	d_\cK(W, Z) &\le	C( T_1^{-1}  + a^{(1)} + a^{(2)} + b^{(0)} + b^{(2)}T_1) \\
	&\le LK \sigma_N^{-1} + C L^2  \sigma_N^{-2} N \sum_{j=K}^{N-1} R(j) + C C_* L^3 K^2 N \sigma_N^{-3} 
	+ C C_* L N \sigma_N^{-1} \hat{R}(K) \\
	&+  C \sqrt{N} \sigma_N^{-2}  L^2 K
	+  C \sqrt{C_*} \sqrt{N} \sigma_N^{-2} L^2 K^{3/2} \biggl\{ 1 + \biggl(  \sum_{j=1}^{N-1} R (j) \biggr)^{1/2} \biggr\}.
\end{align*}
The proof of Theorem \ref{thm:tikhomirov_aux} is complete.
\end{proof}

\section{Proof of Proposition \ref{ex:1b}}\label{s:ex1b}
We start from the verification of Condition 1.
The skew product map is measurable,
$\log \#\mathcal P_\omega=\log 4\in L^1(m)$, and each restriction
$T_\omega|_{I_j}$ is $C^\infty$, monotone, and maps $I_j$ onto $I$ for any $j$.
Moreover,
$|T_\omega^{\prime}(x)|=4$
for all $x\in I,$
and hence the regularity condition holds with $K=1$, and the hyperbolicity
condition holds with $n_1=1$ and $\kappa_1=\kappa_2=4$.
Furthermore, $\mathcal P_\omega^{(n)}$ consists of the $4^n$ intervals of length $4^{-n}$, hence the covering condition
holds with $M(n)=n$ and the positive diameter condition holds with
$\varepsilon_n=4^{-n}$.

\smallskip

We next verify Condition~2. Since $H_\omega$ is a single interval, we have
$h_\omega=1$ and hence $\log h_\omega\in L^1(m)$. The hole condition holds since for every $\omega$ one can choose any partition element
$I_j$ with $j\neq \omega_0$. Then $I_j\cap H_\omega=\varnothing$ and
$T_\omega(I_j)=I$. Moreover, since $T_\omega$ has exactly four inverse
branches, $F_\omega^{(1)}=4.$
On the other hand, since the hole is a partition element and every
surviving branch is full, 
$\xi_\omega^{(1)}=0.$
Hence, 
\[\frac{1}{n_1}\int_\Omega \log
F_\omega^{(1)}\,dm(\omega)
=\log 4>
\log\frac{\kappa_2}{\kappa_1}
+\int_\Omega \log(\xi_\omega^{(1)}+2) dm(\omega)=\log 2,\]
and the growth condition in Condition~2 is satisfied.

\smallskip

We now verify Condition~3. Since the hole coincides with one element of
the partition $\mathcal P_\omega$, the set $X_{\omega,n-1}$ is a union of
entire elements of $\mathcal P_\omega^{(n)}$, which implies
\[
\mathcal P_{\omega,*}^{(n)}
=\{P\in\mathcal P_\omega^{(n)}: P\subset X_{\omega,n-1}\},\]
and every element of $\mathcal P_{\omega,*}^{(n)}$ is a full branch for
$T_\omega^n$. In particular,
$\xi_\omega^{(n)}=0$
for every $n\in\mathbb N$ and $\omega=(\omega_n)_{n\in\mathbb Z}\in\Omega.$
Since $\|(T_\omega^n)^{\prime}\|_\infty=4^n$, we obtain
\[
(9+16\xi_\omega^{(n)})\|(T_\omega^n)^{\prime}\|_\infty^{-1}
=
9\cdot 4^{-n}
\le 9\cdot 2^{-n}.
\]
Hence, the first part of Condition~3 holds with
$\theta=1/2$ and $C_\theta=9.$

To compute $\lambda_\omega$ and $\delta_{\omega,n}$ explicitly, note that
the open transfer operator is given by
\[\mathcal L_\omega u(x)=
\frac{1}{4}\sum_{\substack{0\le k\le 3\\ k\neq \omega_0}}
u\left(\frac{x+k}{4}\right).\]
For $n\ge1$ and $a_0,\dots,a_{n-1}\in\{0,1,2,3\}$ satisfying
$a_j\neq \omega_j$ for all $0\le j\le n-1$, let
\[C_\omega(a_0,\dots,a_{n-1})
=
\bigcap_{j=0}^{n-1} T_\omega^{-j}(I_{a_j}).\]
 Define a probability measure $\nu_{\omega,\infty}$ on
$X_{\omega,\infty}$ by 
\[\nu_{\omega,\infty}\bigl(C_\omega(a_0,\dots,a_{n-1})\bigr)=3^{-n}\]
for every surviving $n$-cylinder. Since
\[\mathcal L_\omega \widehat{C_\omega(a_0,\dots,a_{n-1})}
=\frac{1}{4}
{\widehat{{C_{\sigma\omega}}(a_1,\dots,a_{n-1})}}\]
(recall that $\widehat A$ denotes the indicator function of a measurable set $A$), it follows that
\[\nu_{\sigma\omega,\infty}(\mathcal L_\omega u)=
\frac{3}{4}\nu_{\omega,\infty}(u)\]
for all $u\in BV(I)$. Therefore
$\lambda_\omega={3}/{4}$
for all $\omega\in\Omega.$
Consequently,
\[
{\rm esssup}_{\omega\in\Omega}\frac{\theta}{\lambda_\omega}
=\frac{1/2}{3/4}
=\frac{2}{3}<1.\]

Finally, every good interval in $\mathcal P_{\omega,g}^{(n)}$ is a
surviving $n$-cylinder and hence has $\nu_{\omega,\infty}$-mass exactly $3^{-n}$. Thus
$\delta_{\omega,n}=3^{-n}$
for all $\omega\in\Omega, n\in\mathbb N,$
and therefore
${\rm esssup}_{\omega\in\Omega}
(2\xi_\omega^{(n)}+1)\delta_{\omega,n}^{-1}=3^n<\infty$
for every $n\in\mathbb N.$
This proves Condition~3.

\smallskip

It remains to verify Condition~4. 
We first show that $\phi_\omega\equiv 1$ for all $\omega\in\Omega$.
Define the normalized open transfer operator by
\[\tilde{\mathcal L}_\omega u(x):=\lambda_{\omega}^{-1}\mathcal L_{\omega}u(x)=
\frac{1}{3}\sum_{\substack{0\le k\le 3\\ k\neq \omega_0}}
u\left(\frac{x+k}{4}\right).\]
Since $\mathcal L_{\omega}\phi_{\omega}=\lambda_{\omega}\phi_{\sigma\omega}$, we have 
\[{\rm var}(\phi_{\sigma\omega})={\rm var}(\tilde{\mathcal L}_{\omega}\phi_{\omega})\le \frac{1}{3}\sum_{\substack{0\le k\le 3\\ k\neq \omega_0}}{\rm var}(\phi_{\omega}|_{I_k})\le \frac{1}{3}{\rm var}(\phi_{\omega}).\] Iterating this, we obtain \[{\rm var}(\phi_{\omega})\le 3^{-n}{\rm var}(\phi_{\sigma^{-n}\omega})\ \text{for all}\ n.\]
Since $\omega\rightarrow||\phi_{\omega}||_{\rm BV}$ is tempered, the right-hand side tends to $0$ as $n\to \infty,$ which implies $\phi_\omega$ is constant. By the normalization $\nu_{\omega,\infty}(\phi_\omega)=1$ we have $\phi_\omega \equiv 1$.
Thus, the former part of Condition~4 holds with $C_\phi=1$.

We now establish the uniform spectral gap estimate. 
Write
\[\widetilde{\cL}_\omega^n u(x)
=(\lambda_\omega^n)^{-1}\cL_\omega^n u(x)
=3^{-n}\sum_{C}u\bigl(T_{\omega,C}^{-n}(x)\bigr),\]
where the sum runs over the $3^n$ surviving $n$-cylinders $C=C_\omega(a_0,\dots,a_{n-1})$.
Since $T$ is affine on each branch with $|T'|=4$, the inverse branch
$T_{\omega,C}^{-n}$ maps $I$ affinely onto $C$.
As $x$ ranges over $I$, the point $T_{\omega,C}^{-n}(x)$ ranges over $C$, and therefore
\begin{equation}\label{eq:osc_branch}
\mathrm{osc}\bigl(u\circ T_{\omega,C}^{-n}\bigr)
=\mathrm{osc}_C(u)
\le \mathrm{var}_C(u),
\end{equation}
where $\mathrm{osc}_C(u):=\sup_C u-\inf_C u$ and $\mathrm{var}_C(u)$ is the variation of $u$ restricted to $C$.
Since the surviving $n$-cylinders are pairwise disjoint, we have
\begin{equation}\label{eq:osc_est}
\mathrm{osc}\bigl(\widetilde{\cL}_\omega^n u\bigr)
\le 3^{-n}\sum_{C}\mathrm{osc}_C(u)
\le 3^{-n}\sum_{C}\mathrm{var}_C(u)
\le 3^{-n}\,\mathrm{var}(u).
\end{equation}
On the other hand, the conformal relation $\nu_{\sigma^n\omega,\infty}(\widetilde{\cL}_\omega^n u)=\nu_{\omega,\infty}(u)$ together with the fact that $\nu_{\sigma^n\omega,\infty}$ is a probability measure implies
\[
\inf\bigl(\widetilde{\cL}_\omega^n u\bigr)
\le \nu_{\omega,\infty}(u)
\le \sup\bigl(\widetilde{\cL}_\omega^n u\bigr).
\]
Combining this with \eqref{eq:osc_est} and recalling that $\phi_{\sigma^n\omega}\equiv 1$, we obtain
\begin{align*}
\left\|Q_{\omega,n}(u)\right\|_\infty
&=\left\|\widetilde{\cL}_\omega^n u-\nu_{\omega,\infty}(u)\,\phi_{\sigma^n\omega}\right\|_\infty\\
&=\left\|\widetilde{\cL}_\omega^n u-\nu_{\omega,\infty}(u)\right\|_\infty
\le \mathrm{osc}\bigl(\widetilde{\cL}_\omega^n u\bigr)
\le 3^{-n}\,\mathrm{var}(u)
\le 3^{-n}\,\|u\|_{\mathrm{BV}}
\end{align*}
for all $\omega\in\Omega$, $n\in\mathbb N$, and $u\in\mathrm{BV}(I)$.
Therefore Condition~4 holds with the deterministic constants $D=1$ and $\kappa=1/3$.

\section*{Acknowledgement}
J.~Lepp\"anen was supported by JSPS via the project LEADER.
Y.~Nakajima was supported by JSPS KAKENHI Grant  Number JP25K17282.
Y.~Nakano was supported by JSPS KAKENHI Grant  Number JP23K03188 and JST PRESTO Grant Number JPMJPR25K8.

\bibliographystyle{abbrvnat}
\bibliography{RQE}

\end{document}